\documentclass[10pt]{amsart}

\usepackage[left,pagewise,mathlines]{lineno} 

\allowdisplaybreaks

\usepackage{amsfonts}
\usepackage{amsmath}
\usepackage{amssymb}
\usepackage{amsthm}
\usepackage{graphicx}
\usepackage{enumerate}
\usepackage{multicol}
\usepackage{mathrsfs}
\usepackage[all,cmtip]{xy}
\usepackage[usenames,dvipsnames]{xcolor}

\usepackage{hyperref}

\DeclareMathAlphabet{\mathpzc}{OT1}{pzc}{m}{it}

\newcounter{dummy} \numberwithin{dummy}{section}
\newtheorem{theorem}[dummy]{Theorem}
\newtheorem{corollary}[dummy]{Corollary}
\newtheorem{lemma}[dummy]{Lemma}
\newtheorem{definition}[dummy]{Definition}
\newtheorem{proposition}[dummy]{Proposition}
\theoremstyle{remark}
\newtheorem{remark}[dummy]{Remark}
\newtheorem{example}[dummy]{Example}

\DeclareFontFamily{U}{mathx}{\hyphenchar\font45}
\DeclareFontShape{U}{mathx}{m}{n}{
      <5> <6> <7> <8> <9> <10>
      <10.95> <12> <14.4> <17.28> <20.74> <24.88>
      mathx10
      }{}
\DeclareSymbolFont{mathx}{U}{mathx}{m}{n}
\DeclareFontSubstitution{U}{mathx}{m}{n}
\DeclareMathAccent{\widecheck}{0}{mathx}{"71}
\DeclareMathAccent{\wideparen}{0}{mathx}{"75}

\newcommand{\calF}{\mathcal{F}}
\newcommand{\calH}{\mathcal{H}}
\newcommand{\calL}{\mathcal{L}}
\newcommand{\calR}{\mathcal{R}}
\newcommand{\calV}{\mathcal{V}}

\DeclareMathOperator{\Ann}{Ann}
\DeclareMathOperator{\Sym}{Sym}
\DeclareMathOperator{\id}{id}

\DeclareMathOperator{\pr}{pr}
\DeclareMathOperator{\rank}{rank}
\DeclareMathOperator{\spn}{span}

\DeclareMathOperator{\Ad}{Ad}
\DeclareMathOperator{\ad}{ad}

\DeclareMathOperator{\GL}{GL}
\DeclareMathOperator{\gl}{\mathfrak{gl}}

\DeclareMathOperator{\tensorg}{\mathbf{g}}
\DeclareMathOperator{\tensorh}{\mathbf{h}}

\DeclareMathOperator{\II}{\mathrm{II}}

\DeclareMathOperator{\Hol}{Hol}
\DeclareMathOperator{\Lie}{Lie}
\newcommand{\rnabla}{\mathring{\nabla}}
\DeclareMathOperator{\FGL}{\mathscr{F}^{GL}}

\newcommand\newbullet{{\kern.8pt\displaystyle\centerdot\kern.8pt}}

\numberwithin{equation}{section}

\title[Horizontal holonomy and foliated manifolds]{Horizontal holonomy and foliated manifolds}

\author[Y.~Chitour, E.~Grong, F.~Jean and P.~Kokkonen]{Yacine Chitour, Erlend Grong, Fr\'ed\'eric Jean and Petri Kokkonen}

\address{Universit\'e Paris XI, Laboratoire des Signaux et Syst\`emes (L2S) Sup\'elec, 3 rue Joliot-Curie, 91192
Gif-sur-Yvette, France.}
\email{yacine.chitour@lss.supelec.fr}

\address{Mathematics Research Unit, University of Luxembourg, 6 rue Richard Coudenhove-Kalergi, L-1359 Luxembourg, Luxembourg}
\email{erlend.grong@uni.lu}

\address{UMA, ENSTA ParisTech, Universit\'{e} Paris-Saclay, 828 bd des Mar\'{e}chaux, 91762 Palaiseau, France}
\email{frederic.jean@ensta-paristech.fr}

\address{Helsinki, Finland}
\email{pvkokkon@gmail.com}



\begin{document}

\let\thefootnote\relax\footnotetext{This research was partially supported by the iCODE Institute, research project of the IDEX
Paris-Saclay, by the Grant ANR-15-CE40-0018 of the ANR, and by the Hadamard Mathematics LabEx (LMH) through the grant number ANR-11-LABX-0056-LMH in the ``Programme des Investissements d'Avenir''. It is also supported by the Fonds National de la
  Recherche Luxembourg (AFR 4736116 and OPEN Project GEOMREV).}
\begin{abstract}
We introduce horizontal holonomy groups, which are groups defined using parallel transport only along curves tangent to a given subbundle $D$ of the tangent bundle. We provide explicit means of computing these holonomy groups by deriving analogues of Ambrose-Singer's and Ozeki's theorems. We then give necessary and sufficient conditions in terms of the horizontal holonomy groups for existence of solutions of two problems on foliated manifolds: determining when a foliation can be either (a) totally geodesic or (b) endowed with a principal bundle structure.
The subbundle $D$ plays the role of an orthogonal complement to the leaves of the foliation in case (a) and of a principal connection in case (b).
\end{abstract}

\maketitle

\tableofcontents
\section{Introduction}
Given a foliation $\calF$ of a Riemannian manifold $(M, \tensorg)$ there are several results related to global geometry \cite{Cai83,Cai86,Her60}, nearly K\"ahler manifolds \cite{Nag02}, PDEs on manifolds \cite{BaBo15,BeBo82,KPP97} and probability theory \cite{Elw14} relying on the property that the leaves of $\calF$ are totally geodesic submanifolds. Hence, given a foliation $\calF$ of a manifold $M$, it is natural to ask if $M$ has a Riemannian metric $\tensorg$ that makes the leaves of $\calF$ totally geodesic. Such a metric always exists locally, but global existence is far from being trivial, see e.g. \cite{Sul78,Bri81}. If we in addition require that a given transverse subbundle $D$ is orthogonal to $\calF$, an appropriate Riemannian metric may not even exist locally. We will show that the existence of such a metric can be determined using horizontal holonomy.

The idea of horizontal holonomy consists of considering parallel transport only along curves tangent to a given subbundle $D \subseteq TM$, often referred to as the horizontal bundle, hence the name. Such a holonomy was first introduced for contact manifolds in \cite{FGR97}, partially based on ideas in \cite{Rum94} and generalized later in \cite{HMCK15}. In this paper, we both define horizontal holonomy in greater generality and most importantly provide tools for computing it, in the form of analogues of the theorems of Ambrose-Singer~\cite{AmSi53} and Ozeki~\cite{Oze56}.

Our above mentioned problem of totally geodesic foliations with a given orthogonal complement can now be rewritten in terms of horizontal holonomy as follows. Consider a manifold $M$ whose tangent bundle is a direct sum $TM = D \oplus V$ with~$V$ being an integrable subbundle corresponding to a foliation~$\calF$, and $D$ being a completely controllable subbundle. Let $H \subseteq \GL(\Sym^2 V^*_x)$ be the $D$-horizontal holonomy group at an arbitrary point $x \in M$, defined relative to a vertical connection on $V$. We prove that there exists a Riemannian metric~$\tensorg$ on $M$ such that $D$ is the $\tensorg$-orthogonal complement to $V$ and the leaves of $\calF$ are totally geodesic submanifolds if and only $H$ admits a fixed point which is positive definite as a quadratic form on $V_x$. This question does not only have relevance for geometry but also for the theory on sub-elliptic partial differential operators. To be more precise, let $L$ be a second order partial differential without constant term and consider its symbol $\sigma_L\colon T^*M \to TM$ as the unique bundle map satisfying
\[
df(\sigma_L(dg)) = \frac{1}{2} \left( L(fg) - fLg - g Lf\right), \qquad \text{for any } f,g \in C^\infty(M).
\]
Let us consider the case where $\alpha (\sigma_L \alpha) \geq 0$ and the image of $\sigma_L$ is equal to a proper subbundle $D$ of $TM$;  hence $L$ is not elliptic. The typical example of such an operator $L$ is the sub-Laplacian operator associated with a sub-Riemannian manifold. Finding a totally geodesic foliation $\calF$ which is orthogonal to $D$ enables one to obtain results on the corresponding heat flow of $L$ such as analogues of the Poincar\'e inequality, the Li-Yau inequality and the parabolic Harnack inequality, see e.g. \cite{BaGa11,BKW14,BBG14,GrTh15a,GrTh15b} for details.

The horizontal holonomy of a vertical connection on $V$ can also be related to the existence of a principal bundle structure on $M$. Assume that the leaves of $\calF$ consist of the fibers of a fiber bundle $\pi\colon M \to B$ and that $D$ is a subbundle transversal to $\calF$. We can then establish a link between a trivial horizontal holonomy and the existence of a principal bundle structure of $\pi$ with $D$ as a principal connection.

The structure of the paper is as follows. In Section~\ref{sec:Definition} we give the definition of horizontal holonomy of a general connection $\omega$ on a principal bundle. In Section~\ref{sec:Equiregular}, we first limit ourselves to the case where $D$ is equiregular and bracket-generating and we introduce the main tool for deriving our results, namely two-vector-valued one-forms related to $D$ that we call selectors. In Section~\ref{sec:Computing}, we prove that the horizontal holonomy of $\omega$ is equal to the full holonomy of a modified connection and we show that the Ambrose-Singer and Ozeki theorem are still valid with an adapted modification of the curvature of $\omega$. In both cases, explicit formulas for the modified connection and curvature are given using a selector of $D$. We rewrite our results in the setting of affine connections in Section~\ref{sec:Affine} and consider horizontal holonomy of a general subbundle $D$ in Section~\ref{sec:NotBG} and Section~\ref{sec:GeneralSB}. 

In Section~\ref{sec:Applications}, we apply horizontal holonomy to vertical connections on foliations. In Section~\ref{sec:TGF}, given a foliation $\calF$ and a transversal subbundle $D$, we provide both necessary and sufficient conditions for the existence of a metric $\tensorg$ such that $\calF$ is totally geodesic with orthogonal complement $D$. In Section~\ref{sec:PB} we use horizontal holonomy to determine when a fiber bundle can be endowed with the structure of a principal bundle with a given connection $D$. We note that holonomy in these two cases is related to parallel transport of respectively symmetric tensors and vectors along curves tangent to $D$ and is not related to concepts of holonomy as in \cite{BlHe83,BlHe84}. Since in both cases, the conditions require the computation of horizontal holonomy groups, we give in Subsection~\ref{sec:curvature} explicit formulas for generating sets of the Lie algebra of such groups in terms of curvature operators.
We deal with concrete examples in Section~\ref{sec:examples}. In particular, we give examples of foliations $\calF$ that cannot be made totally geodesic, given a fixed orthogonal complement. We also completely describe the case of one-dimensional foliations.

\subsection{Notation and conventions}
If $Z$ is a section of a vector bundle $\Pi\colon V \to M$, we use $Z|_x$ to denote its value at $x$.
The space of all smooth sections of $V$ is denoted by $\Gamma(V)$. If $V$ is a subbundle of $TM$, $\Gamma(V)$ is considered as a subalgebra of $\Gamma(TM)$. If $X$ is a vector field, then $\calL_X$ is the Lie derivative with respect to $X$. We use $\Sym^2 V$ to denote the symmetric square of $V$. If $E$ and $F$ are vector spaces, then $\GL(E)$ and $\gl(E)$ denote the space of automorphisms and endomorphisms of $E$, respectively and we identify the space of linear maps from $E$ to $F$ with $E^* \otimes F$.

\subsubsection*{Acknowledgements.} The authors thank E. Falbel  for helpful comments and useful insights. 
\section{Horizontal holonomy} \label{sec:HorHol}
\subsection{Definition of horizontal holonomy group} \label{sec:Definition}

Let $M$ be a finite dimensional, smooth and connected manifold, $\pi\colon P\rightarrow M$ a smooth fiber bundle and $\calV = \ker \pi_*$ the corresponding vertical bundle. For $x\in M$, we use $P_x$ to denote the fiber $\pi^{-1}(x)$ over $x$. Let $\calH$ be an arbitrary subbundle of $TP$. An absolutely continuous curve $c\colon[t_0,t_1] \to P$ is said to be \emph{$\calH$-horizontal} if $\dot c(t) \in \calH_{c(t)}$ for almost every
$t \in [t_0, t_1]$.

A subbundle $\calH$ of $TP$ is said to be an \emph{Ehresmann connection} on~$\pi$ if $TP = \calH \oplus \calV$. Here $\pi\colon P \to M$ is considered as a surjective submersion. For every $x \in M$, $v \in T_xM$ and $p$ in the fiber $P_x$, there is a unique element $h_p v \in \calH_p$ satisfying $\pi_* h_p v = v$ called \emph{the $\calH$-horizontal lift} of~$v$. Furthermore, if $\gamma\colon[t_0, t_1] \to M$ is an absolutely continuous curve in $M$ with $\gamma(t_0) = x_0$, a \emph{horizontal lift of~$\gamma$}  is a $\calH$-horizontal absolutely continuous curve $c\colon[t_0,t_1] \to P$ that projects to $\gamma$.
As any horizontal lift $c(t)$ is solution of the ordinary differential equation
\[
\dot c(t) = h_{c(t)} \dot \gamma(t),
\]
$c(t)$ is uniquely determined by its initial condition $c(t_0) \in P_{x_0}$ on an open subinterval of $[t_0,t_1]$ containing $t_0$.
The Ehresmann connection is said to be \emph{complete} if, for every absolutely continuous $\gamma\colon[t_0,t_1] \to M$, all  corresponding horizontal lifts are defined on $[t_0,t_1]$.

A smooth fiber bundle $\pi\colon P\rightarrow M$ is called a principal $G$-bundle if it admits a continuous right action  $P\times G\rightarrow P$ such that the connected Lie group $G$ with Lie algebra~$\mathfrak{g}$ preserves the fibers and acts freely and transitively on them.
For every $\mathfrak{g}$-valued function $f \in C^\infty(P, \mathfrak{g})$, let $\sigma(f)$ be the vector field on $P$ defined by
\begin{equation} \label{SigmaVF}\sigma(f)|_p = \left. \frac{d}{dt} p \cdot \exp_G(tf(p)) \right|_{t=0}, \qquad p \in P. \end{equation}
In particular, for any element $A \in \mathfrak{g}$, we get a corresponding vector field $\sigma(A)$ by considering it as a constant function on $M$. Then $P \times \mathfrak{g} \to \calV, (p, A) \mapsto \sigma(A)|_p$ is an isomorphism of vector bundles.
A \emph{connection form}  $\omega$ on $P$ is a $\mathfrak{g}$-valued one-form~$\omega$ satisfying
\[\omega(\sigma(A)|_p) = A, \qquad \omega(v \cdot a) = \Ad(a^{-1}) \omega(v),\]
for every $A \in \mathfrak{g}$, $p \in P$, $v \in TP$ and $a \in G.$
We say that an Ehresmann connection~$\calH$ on the principal $G$-bundle $\pi\colon P\rightarrow M$ is \emph{principal} if it is invariant under the group action, i.e., if $\calH_p \cdot a = \calH_{p \cdot a}$ for any $p \in P$, $a \in G$. 
An Ehresmann connection is principal if and only if there exists a connection form $\omega$ on $P$ such that $\calH= \ker \omega$. In that case, $\calH$-horizontal curves or lifts will also be referred to as $ \omega$-horizontal. Note that a principal Ehresmann connection is complete.

In what follows, $\calH$ is assumed to be a principal Ehresmann connection on $\pi$ corresponding to a connection form $\omega$. However, non-principal Ehresmann connections will appear elsewhere in the text. For more on Ehresmann connections and principal bundles, we refer to \cite{KMS93}.

Let $\omega$ be a connection form on a principal $G$-bundle $\pi\colon P\rightarrow M$ and let $\calH= \ker \omega$. For every $p \in P$, we use $\mathscr{L}^\omega(p)$ to denote the collection of all $\calH$-horizontal lifts $c\colon[0,1] \to P$ of absolutely continuous loops $\gamma\colon[0,1] \to M$ based in $\pi(p)$ such that $c(0) = p$. \emph{The holonomy group} of $\omega$ at $p$ is then defined as
\[\Hol^\omega(p) = \left\{ a \in G \, : \, c(1) = p \cdot a \text{ for some } c \in \mathscr{L}^\omega(p) \right\}.\]
Since $M$ is connected, the groups $\Hol^\omega(p)$ coincide up to conjugation.

Let us now consider an arbitrary subbundle $D$ of $TM$. We want to introduce a type of holonomy that only considers the loops in $M$ that are $D$-horizontal.
\begin{definition}
For $p\in P$, let $\mathscr{L}^{\omega,D}(p) \subset \mathscr{L}^\omega(p)$ be the collection of $\calH$-horizontal lifts of all $D$-horizontal loops $\gamma\colon[0,1] \to M$ based in~$\pi(p) = x$. The horizontal holonomy group of $\omega$ with respect to $D$ is the group
\[\Hol^{\omega,D}(p) = \{a \in G \, : \, c(1) = p \cdot a \text{ for some } c \in \mathscr{L}^{\omega,D}(p)\}.\]
\end{definition}
If $D$ is completely controllable (i.e., any two points in $M$ can be connected by a $D$-horizontal curve), then the groups
$\Hol^{\omega,D}(p)$ with $p \in P$, coincide up to conjugation.
If $\omega$ and $\widetilde \omega$ are two connections on~$P$, the sets $\mathscr{L}^{\omega,D}(p)$ and $\mathscr{L}^{\widetilde \omega,D}(p)$ may coincide for every $p \in P$ even if the connections are different. Since in this case these connections also have the same horizontal holonomy group with respect to $D$, we introduce the following equivalence relation on connections of $P$.

\begin{definition}
Let $\pi\colon P \to M$ be a principal $G$-bundle and $D$ a subbundle of~$TM$.
Two connection forms $\omega$ and $\widetilde \omega$ are called $D$-equivalent if
\[\omega(v) = \widetilde \omega(v), \quad \text{ for any } v \in TP \text{ satisfying } \pi_* v \in D,\]
and we write $[\omega]_D$ for the equivalence class of the connection form $\omega$.
\end{definition}
Any two $D$-equivalent connection forms $\omega$ and $\widetilde \omega$ have the same horizontal lifts to $P$ of $D$-horizontal curves and hence $\Hol^{\omega,D}(p) = \Hol^{\widetilde \omega,D}(p)$ for every $p \in P$.

\begin{remark} \label{re:connected}
\begin{enumerate}[\rm (a)]
\item Rather than introducing the above equivalence classes, we could have considered \emph{partial connections} such as in \cite{FGR97}: given a principal $G$-bundle $\pi\colon P \to M$ and a subbundle $D$ of $TM$, a (principal) partial connection over~$D$ is a subbundle~$\mathcal{E}$, invariant under the action of~$G$, such that~$\pi_* $ maps~$\mathcal{E}$ on~$D$ bijectively on every fiber. For every equivalence class $[\omega]_D$, we obtain a partial connection by $\mathcal{E} \colon= (\pi_*)^{-1}(D) \cap \ker \omega$. Conversely, following the argument of \cite[Theorem~2.1]{KoNo63}, one proves that any partial connection can be extended to a full connection on~$\pi$. Hence, there is a one-to-one correspondence between partial connections and $D$-equivalence classes. For us, the language of $D$-equivalence classes will be more convenient.
\item  For any connection form $\omega$, the identity component of $\Hol^\omega(p)$ is obtained by horizontally lifting all contractible loops based at $\pi(p)$. For horizontal holonomy, we have a similar description. For any  loop $\gamma\colon[0,1] \to M$ based in $x$, we say that it is \emph{$D$-horizontally contractible} if $\gamma$ is a $D$-horizontal loop and if there exists a homotopy $[0,1] \times [0,1] \to M$, $(s,t) \mapsto \gamma^s(t)$ such that $\gamma^0(t) = x$, $\gamma^1(t) = \gamma(t)$, $\gamma^s(0) = \gamma^s(1) = x$ and $t \mapsto \gamma^s(t)$ is a $D$-horizontal curve for any~$s \in [0,1]$. The identity component of $\Hol^{\omega,D}(p)$ is obtained by horizontally lifting $D$-horizontally contractible loops. If $D$ is bracket-generating, (i.e., if~$TM$ is spanned by vector fields with values in~$D$ and their iterated brackets)
then a~$D$-horizontal loop is $D$-horizontally contractible if and only if it is contractible (see \cite{Sarychev1990} and \cite[Theorem 1]{Ge1993}). As a consequence, the identity component of $\Hol^{\omega,D}(p)$ is obtained by horizontally lifting contractible $D$-horizontal loops. On the other hand, such a property may not hold when $D$ is not bracket-generating, as the following example shows. Consider $\mathbb{R}^4$ with coordinates $(x,y,z,w)$ and let $D$ be the span of $\frac{\partial}{\partial x}$ and $\frac{\partial}{\partial y} + x (w\frac{\partial}{\partial z} - z\frac{\partial}{\partial w})$. Fix a point $(x,y,z,w)$ with $(z,w)\neq (0,0)$. Then all $D$-horizontal loops starting from this point are contained in a manifold diffeomorphic to $\mathbb{R}^2 \times S^1$ and some of them are contractible but not $D$-horizontally contractible.

\item The definition of horizontal holonomy does not change if we define $\mathscr{L}^{\omega,D}(p)$ to be the collection of horizontal lifts of all loops based in $\pi(p)$ that are both $D$-horizontal and smooth, see \cite[Theorem 2.3]{bel96} and \cite[last sentence]{Grasse1990}.

\end{enumerate}
\end{remark}

\subsection{Equiregular subbundles and selectors} \label{sec:Equiregular}
In this paragraph,
we assume that the subbundle $D$  of $TM$ is \emph{equiregular} and \emph{bracket-generating} and the corresponding definitions are given next.
\begin{definition}\label{def0}
Let $D$ be a subbundle of the tangent bundle $TM$ of a connected manifold $M$.
\begin{enumerate}[$\bullet$]
\item We say that $D$ is \emph{equiregular} of step $r$ if there exist a flag of subbundles of $TM$
\begin{equation} \label{flag} 0 = D^0 \subsetneq D^1 = D \subsetneq D^2 \subsetneq \cdots \subsetneq D^r ,\end{equation}
such that $D^r$ is an integrable subbundle and such that $D^k$ is the span of vectors fields with values in $D$ and their iterated brackets of order less than~$k-1$ for any~$2 \leq k \leq r$.
\item We say that $D$ is \emph{bracket-generating} if $TM$ is spanned by vector fields with values in $D$ and their iterated brackets.
\item We say that $D$ is \emph{completely controllable} if any two points in $M$ can be connected by a $D$-horizontal curve.
\end{enumerate}
\end{definition}
From the above definitions, an equiregular subbundle $D$ is bracket-generating if and only if $D^r = TM$. Furthermore, $D$ is completely controllable if it is bracket-generating \cite{Cho39,Ras38}. We give some examples to illustrate the above definitions.

\begin{example}
\begin{enumerate}[\rm (a)]
\item A subbundle is integrable if and only if it is equiregular of step~$1$.
\item Consider $\mathbb{R}^3$ with coordinates $(x,y,z)$. Let $\phi\colon \mathbb{R} \to \mathbb{R}$ be a real valued smooth function and define $D$ as the span of $\frac{\partial}{\partial x}$ and $\frac{\partial}{\partial y} + \phi(x) \frac{\partial}{\partial z}$.
\begin{enumerate}[\rm (i)]
\item If $\phi(x) = x$, then $D$ is bracket-generating and equiregular of step~$2$.
\item If $\phi(x) = x^2$, then $D$ is bracket-generating, but not equiregular, since $\spn \{ X, Y, [X,Y]\}$ is not of constant rank, and so $D^2$ is not well-defined.
\item If $\phi(0) = 0$ and $\phi(x) = e^{-1/x^2}$ for $x \neq 0$, then $D$ is completely controllable but neither bracket-generating nor equiregular.
\end{enumerate}
\item Consider $\mathbb{R}^4$ with coordinates $(x,y,z,w)$ and let $D$ be the span of $\frac{\partial}{\partial x}$ and $\frac{\partial}{\partial y} + x \frac{\partial}{\partial z}$. Then $D$ is equiregular of step~$2$, but not bracket-generating.
\item If $D$ has rank greater or equal to~$2$, a generic choice of a subbundle $D$ of $TM$ is bracket generating in the sense of \cite[Proposition~2]{Mon93}.

\end{enumerate}
\end{example}
We consider here the case when $D$ is equiregular (of step $r\in \mathbb{N}$) and   bracket-generating. The remaining cases are addressed in Sections~\ref{sec:NotBG} and \ref{sec:GeneralSB}.

For $0 \leq k \leq r$, we use $\Ann(D^k) \subseteq T^*M$ to denote the subbundle of $T^*M$ consisting of the covectors that vanish on $D^k$. In particular, $\Ann(D^r)$ reduces to the zero section of $T^*M$. The following definition introduces the main technical tool in order to formulate our results on horizontal holonomy groups.
\begin{definition} \label{def:Selector}
Let $D$ be a bracket-generating, equiregular subbundle of step~$r$ with the corresponding flag given as in \eqref{flag}. We say that a two-vector-valued one-form $\chi \in \Gamma(T^*M \otimes \bigwedge^2 TM)$ is a \emph{selector of $D$} if it satisfies the following two assumptions.
\begin{enumerate}[\rm (I)]
\item For every $0 \leq k \leq r-1$, $\chi(D^{k+1}) \subseteq \bigwedge^2 D^k \subseteq \bigwedge^2 TM.$
\item For every $1 \leq k \leq r-1$ and one-form $\alpha$ with values in $\Ann(D^k)$ and every vector $v \in D^{k+1}$, we have
\[\alpha(v) = - d\alpha(\chi(v)).\]
\end{enumerate}
\end{definition}
Taking $k=0$ in Item (I) yields that any selector must satisfy $\chi(D) = 0$. If $\chi$ is a selector, we use $\iota_{\chi}$ to denotes its transpose or contraction operator, i.e., for every vector $v$ and two-covector $\eta$ one has $(\iota_{\chi} \eta)(v) \colon= \eta(\chi(v))$.

The next lemma provides basic properties associated with selectors.
\begin{lemma} \label{lemma:Selector}
\begin{enumerate}[\rm (a)]
\item A bracket-generating equiregular subbundle $D$ admits at least one selector.
\item The set of selectors of $D$ in $\Gamma(T^*M \otimes \bigwedge^2 TM)$ is an affine subspace. In fact,  if $\chi^0$ is a selector of $D$, then $\{ \chi - \chi^0 \, : \, \chi \text{ is a selector of }D \}$ is a $C^\infty(M)$-module.
\item Let $\beta$ and $\eta$ be a one-form and a two-form on $M$, respectively. Let $\chi$ be a selector of $D$. Then there exists a unique one-form $\alpha$ satisfying the system of equations
\begin{equation} \label{extBeta}
\alpha|_D = \beta|_D, \qquad \iota_\chi d \alpha = \iota_\chi \eta,
\end{equation}
and $\alpha$ is given by
\begin{equation} \label{solBeta}
\alpha = (\id + \iota_\chi d)^{r-1} \beta  - \iota_\chi \sum_{j=0}^{r-2} \binom{r-1}{j+1} (d \iota_\chi)^j \eta.
\end{equation}
\item Let $d_\chi\colon \Gamma(T^*M)\rightarrow  \Gamma(\bigwedge^2 TM)$
be defined by $d_\chi \colon= d(\id + \iota_\chi d)^{r-1}$. Then, for every one form $\beta$, $d_\chi\beta$ only depends on $\beta|_D$ and $d_\chi\beta=0$ if and only if there exists a one-form $\tilde{\beta}$ such that
\[\tilde{\beta}|_D = \beta|_D, \qquad d\tilde{\beta} = 0.\]
\end{enumerate}
\end{lemma}

We provide an example of selectors before giving the argument of the above lemma.
\begin{example} \label{ex:Contact}
For $n \geq 1$, consider $\mathbb{R}^{2n+1}$ with coordinates $(x_1, \dots, x_n, y_1, \dots, y_n, z)$. For every $1\leq j\leq n$, define $X_j = \frac{\partial}{\partial x_i}$ and $Y_j = \frac{\partial}{\partial y_i} + x_i \frac{\partial}{\partial z}$ and let $D$ be the span of these vector fields. The subbundle $D$ is then bracket-generating and equiregular of step~$2$. Furthermore, for every $1\leq k\leq n$, the two-vector-valued one-forms $\chi^k$ defined, for every $1\leq j\leq n$, by
\[\chi^k (X_j) = 0, \quad \chi^k(Y_j)=0 \quad \text{and} \quad \chi^k(\tfrac{\partial}{\partial z}) = X_k \wedge Y_k,\]
are selectors of~$D$. The collection of all selectors of~$D$ is
\[
 \chi^1 + \mathrm{span}_{C^\infty(M)} \{ \chi^j- \chi^1, \, 2\leq j\leq n\}=
 \left\{ \sum_{k=1}^n f^k \chi^k \, : \, f^k \in C^\infty(M), \, \sum_{k=1}^n f^k \equiv 1 \right\}.
 \]
\end{example}

\begin{remark}
The reason for our choice of the term ``selector'' is the following. Let $Z$ be a vector field with values in $D^{k+1}$ for some $k=0, \dots, r-1$. By definition, we can write $Z$ using vector fields with values in $D^k$ and first order Lie brackets of vector fields with values in the same subbundle. However, such a decomposition is not unique. The idea is that a selector gives us a way of selecting one of these representations. That is, if $\chi(Z) = \sum_{i=1}^l X_i \wedge Y_i,$ Items (I) and (II) in Definition~\ref{def:Selector} ensure that we can write
\[Z = \sum_{i=1}^l [X_i, Y_i] + Z_2,\]
where the vector fields $X_i, Y_i$ and $Z_2 = Z - \sum_{i=1}^k [X_i, Y_i]$ take values in $D^k$.
\end{remark}

\begin{proof}[Proof of Lemma~\ref{lemma:Selector}]
\begin{enumerate}[\rm (a)]
\item Endow $M$ with a Riemannian metric~$\tensorg$. Let $E^k$ denote the $\tensorg$-orthogonal complement of $D^{k-1}$ in $D^k$ for $1 \leq k \leq r$. In other words
\[D = E^1, \qquad D^2 = E^1 \oplus_\perp E^2, \qquad \dots, \qquad D^r = TM = E^1 \oplus_\perp \cdots \oplus_\perp E^r. \]
For $1 \leq k \leq r$, let $\pr_{E^k}$ denote the $\tensorg$-orthogonal projection onto $E^k$ and set $\pr_{E^{r+1}}$ to be equal to the zero-map. We next define a vector-valued two-form $\Phi\colon \bigwedge^2 TM \to TM$ as follows. Let $X$ and $Y$ be two vector fields with values in $E^i$ and $E^j$ respectively with $i \leq j$ and let $x \in M$. We write $v = X|_x$ and $w =Y|_x$. We set
\[\Phi(v ,w) = \left\{ \begin{array}{ll}
0 & \text{if } i \geq 2, \\
\pr_{E^{j+1}} [X,Y]|_x  & \text{if } i =1. \end{array} \right.\]
Since $\pr_{E^{j+1}} [X,Y]|_x$ does not depend on the choices of sections $X$ and $Y$ of respectively $E^1$ and $E^j$, the vector $\Phi(v,w)$ is well defined. The image of $\Phi$ is $E^2 \oplus \cdots \oplus E^r$. Define $\chi\colon TM \to \bigwedge^2 TM$ such that $\chi$ vanishes on $E^1$ and for any $w \in E^k$, $2 \leq k \leq r$, $\chi(w)$ is the unique element in $\bigwedge^2 TM$ satisfying
\[\Phi(\chi(w)) = w, \qquad \chi(w) \in (\ker \Phi)^\perp,\]
with the latter denoting the $\tensorg$-orthogonal complement of the kernel of $\Phi$ in $\bigwedge^2 TM$. From this definition, Item (I) follows readily. Furthermore, let~$X$ and~$Y$ be two arbitrary vector fields with values in $E^1$ and $E^j$ respectively, with $j <r$. If $Z = \Phi(X,Y)$ and $\alpha$ is a one-form vanishing on $D^j = E^1 \oplus \cdots \oplus E^j$, then
\[- d\alpha(\chi(Z)) = - d\alpha(X,Y) = \alpha([X,Y]) = \alpha(\Phi(X,Y)) = \alpha(Z),\]
so (II) is satisfied as well.
\item If $\chi_1$ and $\chi_2$ are two selectors of $D$, then from Definition~\ref{def:Selector}, we have that $\xi = \chi_1 - \chi_2$ is a map that satisfies $\xi(D^{k+1}) \subseteq \bigwedge^2 D^k, k=0, \dots, r-1$ and for any $\alpha \in \Gamma(\Ann D^k)$ and $w \in D^{k+1}$, $k =1, \dots, r-1$, we have
\[d \alpha(\xi(w)) = 0.\]
Clearly, the space of all such $\xi$ is closed under addition and multiplication by scalars or functions.
\item
We start by showing uniqueness of a solution of \eqref{extBeta}. Thanks to the linearity of the equations of \eqref{extBeta}, it amounts to prove that $\alpha=0$ is the unique solution to \eqref{extBeta} when $\beta = 0$ and $\eta =0$ . Such an $\alpha$ must take values in $\Ann(D)$, meaning that, for every $w \in D^2$, we have $d\alpha(\chi(w)) = 0 = - \alpha(w)$, and so $\alpha$ must vanish on $D^2$ as well. By iterating this reasoning, it follows that $\alpha = 0$. 

As regards the existence of a solution of \eqref{extBeta}, the linearity of the equations of \eqref{extBeta}  implies that it is enough to consider two cases, namely (i) $\beta = 0$ and (ii) $\eta=0$. We start with Case (i). Since $\iota_\chi \eta$ vanishes on $D$, it follows that $(\id + \iota_\chi d) \iota_\chi \eta$ vanishes on $D^2$ by Definition~\ref{def:Selector}~(II). Iterating this argument, we obtain
\[(\id + \iota_\chi d)^{r-1} \iota_\chi \eta = \sum_{j=0}^{r-1} \binom{r-1}{j} (\iota_\chi d)^j \iota_\chi \eta = 0.\]
Hence, $\iota_\chi \eta = - \iota_\chi d \sum_{j=0}^{r-2} \binom{r-1}{j+1} (\iota_\chi d)^j \iota_\chi \eta$ and so we can take \[\alpha = -  \sum_{j=0}^{r-2} \binom{r-1}{j+1} (\iota_\chi d)^j \iota_\chi \eta = -  \iota_\chi \sum_{j=0}^{r-2} \binom{r-1}{j+1} (d \iota_\chi)^j \eta \] a solution to \eqref{extBeta}. Note that $\alpha$ vanishes on $D$ as required, since $\chi$ vanish on~$D$.

We next turn to Case (ii), i.e., we assume that $\eta = 0$ in \eqref{extBeta}. Define $\alpha^1 = \beta$ and $\alpha^{k+1} = (\id + \iota_\chi d) \alpha^k$ for $k=1, \dots, r-1$. We show by induction on $k\geq 1$ that $\alpha^k(v) = \beta(v)$ for  $v \in D$ and $(\iota_\chi d \alpha^k) (w) = 0$ for $w \in D^k$. The conclusion trivially holds for $k =1$ since $\chi$ vanishes on $D^1$. Furthermore, for every $v \in D$, one has $\alpha^{k+1}(v) = \alpha^k(v) + \iota_\chi d \alpha^k(v) = \alpha^k(v).$ We complete the induction step by observing that
\[\iota_\chi d \alpha^{k+1} = (\id+ \iota_\chi d) \iota_\chi d\alpha^k,\]
vanishes on $D^{k+1}$ if $\iota_\chi d \alpha^k$ vanishes on $D^k$. The desired solution $\alpha$ is simply~$\alpha^r$.
\item If $\beta_1|_D = \beta_2|_D$, then $\beta_1 - \beta_2$ is a one-form with values in $\Ann(D)$ and one has that \[(\id + \iota_\chi d)^{r-1} (\beta_1 - \beta_2) = 0.\]
Hence $d_\chi \beta_1 = d_\chi \beta_2$.

Consider a closed one form $\tilde{\beta}$ and a one-form $\beta$ such that $ \beta |_D=\tilde{\beta}|_D$.
Then $d_\chi \beta = d_\chi \tilde{\beta} = (\id + d \iota_{\chi})^{r-1} d\tilde{\beta} = 0$. Conversely, if $\beta$ is a one-form such that $d_\chi \beta = 0$, then $\tilde{\beta}: = (\id + \iota_\chi d)^{r-1} \beta$ clearly satisfies the two equations $d \tilde{\beta} = 0$ and $\tilde{\beta}|_D = \beta |_D$.
\end{enumerate}
\end{proof}
We next extend the conclusion of Lemma~\ref{lemma:Selector}~(c) to the context of forms taking values in a vector bundle. For that purpose, consider a vector bundle $E\to M$ with an affine connection $\nabla$. \emph{The exterior covariant derivative} $d^\nabla$  is defined as
follows: for every $k$-form $\eta\in\Gamma(\bigwedge^k T^*M \otimes E)$ with $k \geq 0$,
\begin{enumerate}[\rm (a)]
\item If $k = 0$, then $d^\nabla \eta = \nabla_{\centerdot} \eta$,
\item If $\beta$ is a real-valued form, then $d^\nabla(\eta \wedge \beta) = (d^\nabla \eta) \wedge \beta + (-1)^k \eta \wedge d\beta$.
\end{enumerate}
Then, the conclusion of Lemma~\ref{lemma:Selector}~(c) still holds true for forms taking now values in any vector space if one replaces the exterior derivative $d$ with the exterior covariant derivative $d^\nabla$. Indeed, if $\alpha$ is an $E$-valued one-form vanishing on $D^k$, then for any $w \in D^{k+1}$ and selector $\chi$, we have
\[d^\nabla \alpha(\chi(w)) = - \alpha(w).\]
Hence, we can use the same argument as in the proof of Lemma~\ref{lemma:Selector}~(c) to obtain a formula for the unique solution $\alpha$ to the equation $\alpha|_D = \beta|_D$ and $\iota_\chi d^\nabla \alpha = \iota_\chi \eta$ for given $\beta$ and $\eta$. In fact, we can get the following more general result by using the same approach.
\begin{corollary} \label{cor:Covariant}
Let $\chi$ be a selector of $D$, $\Pi\colon E \to M$ a vector bundle over~$M$ and $\beta$, $\eta$ respectively a one-form and a two form taking values in $E$. Consider an operator $L\colon \Gamma(T^*M \otimes E) \to \Gamma(\bigwedge^2 T^*M \otimes E)$
such that, for $1\leq k\leq r-1$ and $\alpha \in \Gamma(T^*M \otimes E)$ vanishing on $D^k$, one has that $(\id + \iota_\chi L) \alpha$ vanishes on $D^{k+1}$. Then the unique solution $\alpha$ to the system of equations $ \alpha|_D = \beta|_D$, $\iota_\chi L \alpha = \iota_\chi \eta$
is given by
\[
\alpha = (\id + \iota_\chi L)^{r-1} \beta  - \iota_\chi \sum_{j=0}^{r-2} \binom{r-1}{j+1} (L \iota_\chi)^j \eta.
\]
Furthermore, if we define $L_\chi  \colon= L(\id + \iota_\chi L)^{r-1}$, then $L_\chi \alpha$ only depends on $\alpha|_D$ and it vanishes
$L_\chi \alpha = 0$ if and only if there exists some one-form $\beta$ with
\[\alpha|_D = \beta|_D , \qquad L \beta = 0.\]
\end{corollary}

\subsection{Computation of horizontal holonomy groups} \label{sec:Computing}
A central result for the characterization of the holonomy of a connection $\omega$ in principal bundles is the Ambrose-Singer Theorem, which essentially says that the Lie algebra of $\Hol^\omega$ can be computed from the \emph{curvature form} $\Omega$ of $\omega$, see \cite{AmSi53} and \cite[Theorem~8.1]{KoNo63}. Recall that in the case of infinitesimal holonomy or in the analytic framework, the Ambrose-Singer Theorem admits a sharpened form established by Ozeki \cite{Oze56}. In this section, we provide versions of Ambrose-Singer and Ozeki Theorems describing the horizontal holonomy group of equiregular subbundles, and which rely on an adapted curvature form that we introduce below.

For that purpose, we consider the following notations.
Let $\pi\colon P \to M$ be a principal $G$-bundle where $\mathfrak{g}$ denotes the Lie algebra of $G$.
We define a bracket of $\mathfrak{g}$-valued forms on $P$ as next: if $\alpha$ and $\beta$ are real valued forms and $A, B \in \mathfrak{g}$, then
\[[\alpha \otimes A , \beta \otimes B]: = (\alpha \wedge \beta)\otimes [A,B] .\]
In particular, if $\eta$ is a $\mathfrak{g}$-valued one-form, then $[\eta, \eta](v,w) = 2 [\eta(v), \eta(w)].$

A function $f$ (a form $\eta$ respectively) on $P$ with values in $\mathfrak{g}$ is called \emph{$G$-equivariant} if it satisfies
\[f(p \cdot a) = \Ad(a^{-1}) f(p)\quad (\eta(v_1 \cdot a, \cdots, v_k \cdot a) = \Ad(a^{-1}) \eta(v_1, \dots, v_k) \text{ respectively}).
\]
Consider the vector bundle $\Ad P \to M$ defined as the quotient $(P \times \mathfrak{g})/G$ with respect to the right action of $G$ given by $(p,A) \cdot a \colon= (p \cdot a, \Ad(a^{-1}) A).$ Any section $s \in \Gamma(\Ad P)$ defines a unique $G$-equivariant map $s^{\wedge}\colon P \to \mathfrak{g}$ such that $s(\pi(p)) = (p, s^\wedge(p))/G$. In that way, one can associate with a connection form $\omega$ on $P$ an affine connection $\nabla^\omega$ on $\Ad P$ by letting $\nabla_X^\omega s$ be the section of $\Ad P$ corresponding to the $G$-equivariant function $ds^\wedge(hX)$. Here, $X$ is a vector field on $M$ and $hX$ denotes its $\omega$-horizontal lift defined by $hX|_p = h_p X|_{\pi(p)}$.

We have a similar identification between $\Ad P$-valued forms and $G$-equivariant horizontal forms. Write $\calH= \ker \omega$ and $\calV = \ker \pi_*$. We say that a form on $P$ is \emph{horizontal} if it vanishes on $\calV$. Any $\Ad P$-valued form $\eta \in \Gamma(\bigwedge^k T^*M \otimes \Ad P)$ corresponds uniquely to a horizontal $G$-equivariant form $\eta^\wedge$ by $\eta(v_1, \dots v_k) = (p, \eta^\wedge(h_p v_1, \dots, h_p v_k))/G$ where $v_j \in T_x M, j=1,\dots, k$ and $p \in P_x$. From this definition, it follows that $d^{\nabla^\omega}\eta$ corresponds to $\pr_{\calH}^* d\eta^\wedge$. Moreover, if $\alpha$ and $\beta$ are $\Ad(P)$-valued forms, we will use $[\alpha, \beta]$ for the $\Ad(P)$-valued form corresponding to~$[\alpha^\wedge, \beta^\wedge]$.

A special $\Ad P$-valued form is the curvature form $\Omega$ of the connection $\omega$, corresponding to the equivariant horizontal two-form $\Omega^\wedge = \pr_{\calH}^* d\omega.$ Note that $\Omega^\wedge(v,w) = d\omega(v,w) + \frac{1}{2} [\omega, \omega].$
Moreover, for every vector fields $X,Y$ on $M$, one has that
\begin{equation}\label{curF0}
[hX, hY] - h[X,Y]  = -\Omega^\wedge(hX,hY).
\end{equation}

The next proposition describes the horizontal holonomy group of a connection $\omega$ with respect to a subbundle $D$ as the holonomy of an adapted connection.
\begin{proposition} \label{th:UsualHol}
Let $\pi\colon P \to M$ be a principal bundle over $M$. Let $D$ be an equiregular and bracket-generating subbundle of $TM$.
\begin{enumerate}[\rm (a)]
\item Let $\omega$ be any connection form on $\pi$ with corresponding curvature form $\Omega$. If for some selector $\chi$ of $D$, we have
\begin{equation} \label{CurvChiZero} \iota_\chi \Omega = 0,\end{equation}
then $\Hol^{\omega,D}(p) = \Hol^\omega(p)$ for any $p \in P$.
\item For any connection form $\omega$ and selector $\chi$ of $D$, there exists a unique connection $\widetilde \omega \in [\omega]_D$ with curvature satisfying \eqref{CurvChiZero}.
\end{enumerate}
As a consequence, for any connection form $\omega$, there exists a connection form $\widetilde \omega$ such that
\[\Hol^{ \omega,D}(p) = \Hol^{\widetilde \omega,D}(p) = \Hol^{\widetilde \omega}(p), \qquad \text{for any } p \in P.\]
In particular, $\Hol^{\omega,D}(p)$ is a Lie group.
\end{proposition}
The proof relies on Corollary~\ref{cor:Covariant} and on the following lemma. We first give some extra notation. For every subset $\mathscr{A}$ of the Lie algebra $(\Gamma(TP),[\newbullet,\newbullet])$, we use $\Lie \mathscr{A}$ and $\Lie_p \mathscr{A}$, $p\in P$, to denote respectively the Lie algebra generated by $\mathscr{A}$ and the subspace of $T_pP$ made of the evaluations at $p$ of the elements of $\Lie \mathscr{A}$.

\begin{lemma} \label{lemma:control}
Let  $\calH$ and $\widehat \calH$ be two subbundles of $TP$. For $p \in P$, define the orbit $\mathscr{O}_{p}$ of $\calH$ at $p$ as the sets of points in $P$ that can be reached from $p$ by $\calH$-horizontal curves. Define in the same way the orbit $\widehat {\mathscr{O}}_{p}$  of $\widehat \calH$ at $p$. If
\begin{equation} \label{LieAlg} \Lie_p \Gamma(\widehat \calH)  \subseteq \Lie_p \Gamma(\calH), \qquad \text{for every } p \in P,\end{equation}
then $\widehat {\mathscr{O}}_{p}  \subseteq \mathscr{O}_{p}$ also holds for every $p \in P$. If equality holds true in \eqref{LieAlg}, then $\widehat {\mathscr{O}}_p = \mathscr{O}_p.$
\end{lemma}
Note that for every points $p_0$ and $p_1$ of $P$, the orbits $\mathscr{O}_{p_0}$ and $\mathscr{O}_{p_1}$ are either disjoint or coincide.
\begin{proof}
From the Orbit theorem, see e.g. \cite[Theorem~5.1]{AgSa04}, one gets that,  for every $p \in P$, $\mathscr{O}_{p}$ is a connected immersed submanifold of $P$. Furthermore, by \cite[Corollary~5.1]{AgSa04}, we have that for every $p_0 \in P$ and $p \in \mathscr{O}_{p_0}$, $\Lie_p \Gamma(\calH) \subseteq T_p \mathscr{O}_{p_0}$. It follows that for every $p \in \mathscr{O}_{p_0}$,
\[\widehat \calH_p \subseteq \Lie_p \Gamma(\widehat \calH) \subseteq \Lie_p \Gamma(\calH) \subseteq T_p \mathscr{O}_{p_0}.\]
As a consequence, $\widehat \calH|_{\mathscr{O}_{p_0}} \subset T \mathscr{O}_{p_0}$. Hence, for every $\widehat \calH$-horizontal curve $c\colon[0,1] \to P$ and $t \in [0,1]$, there exists a connected neighborhood $U_t$ of $t$ in $[0,1]$ such that $c(U_t) \subseteq \mathscr{O}_{c(t)}$. Since $[0,1]$ is compact, we can pick a finite number of point $0 = t_0 < t_1 < t_2 < \cdots < t_k \leq 1$, such that $\{ U_{t_j} \}_{j=1}^n$ is an open covering of $[0,1]$. Since for $j=0, \dots, k-1$, $U_j$ and $U_{j+1}$ are not disjoint, it must follow that the orbits $\mathscr{O}_{c(t_j)}$ all coincide with $\mathscr{O}_{c(t_0)} = \mathscr{O}_{c(0)}.$ Hence every $\widehat \calH$-horizontal curve $c$ with $c(0) = p$ is contained in $\mathscr{O}_p$, implying that $\widehat {\mathscr{O}}_p \subseteq \mathscr{O}_p$.
\end{proof}
We now turn to the proof of Proposition~\ref{th:UsualHol}.
\begin{proof}[Proof of Proposition~\ref{th:UsualHol}]
\begin{enumerate}[\rm (a)]
\item
For $1\leq k\leq r$, consider the subbundles $\mathcal{E}^k$ of $TP$ defined by
\[
\mathcal{E}^k = \{ h_p v \, : \, v \in D^k, p \in P\}.
\]
For every $p \in P$, let $\mathscr{O}_p$ and $\mathscr{O}^k_p$ denote the orbits of $\calH$ and $\mathcal{E}^k$ at $p$ respectively. From the definition of holonomy, it follows that
\[
\Hol^\omega(p) = \{ a \in G \, : \, p \cdot a \in \mathscr{O}_p \}.
\]
The same identity holds for $\Hol^{\omega,D}(p)$ with $\mathscr{O}_p$ substituted by $\mathscr{O}^1_p$. Hence, $\Hol^{\omega}(p) = \Hol^{\omega,D}(p)$ if  $\mathscr{O}^1_p =  \mathscr{O}_p$. We next show that these equalities hold true if equality
 \eqref{CurvChiZero} holds. We first prove that $\Lie \Gamma(\mathcal{E}^k) = \Lie\Gamma( \mathcal{E}^{k+1})$ for $1\leq k\leq r-1$.

Let then $1\leq k\leq r-1$ and notice that one has the obvious inclusion $\Lie \Gamma(\mathcal{E}^k) \subseteq \Lie \Gamma(\mathcal{E}^{k+1})$. Equality follows if $hZ\in  \Lie \Gamma(\mathcal{E}^k)$
for any vector field  $Z$ with values in $D^{k+1}$. Pick such a $Z$ and
let $X_1, \dots, X_l,$ and  $Y_1, \dots, Y_l$ be any choice of vector fields with values in $D^k$ such that $\chi(Z) = \sum_{i=1}^l X_i \wedge Y_i$. From the definition of $\chi$, it follows that $Z = \sum_{j=1}^l [X_j, Y_j] + Z_2$ where $Z_2$ is some vector field with values in $D^k$. Using~\eqref{curF0} and~\eqref{CurvChiZero}, we deduce that
\begin{equation} \label{hZ} hZ = \sum_{j=1}^l [hX_j,hY_j] + hZ_2 \in \Lie \Gamma(\mathcal{E}^k).
\nonumber \end{equation}
We finally get that $\Lie \Gamma(\mathcal{E}^1) = \Lie \Gamma(\mathcal{E}^r) = \Lie \Gamma(\calH)$ and conclude the argument by using Lemma~\ref{lemma:control}.

\item
Let $\widetilde \omega$ be a connection form $\widetilde \omega = \omega + \alpha^\wedge$ with $\alpha^\wedge$ being an equivariant horizontal one-form. Write the $\widetilde \omega $-horizontal lift   as $\widetilde h$ and let $\widetilde \Omega$ be the curvature form of $\widetilde \omega$. By definition,
one has $\widetilde h_p w = h_p w - \sigma(\alpha^\wedge( h_p w))$ for any $p \in P_x$, $w \in T_xM$, $x \in M$. Furthermore, for any vector fields $X$ and $Y$ on $M$, we have
\begin{linenomath}
\begin{align*}
& \quad \widetilde \Omega^\wedge(\widetilde hX, \widetilde hY) = d \omega(\widetilde hX, \widetilde hY) + d\alpha^\wedge(\widetilde hX, \widetilde hY) \\
& = - [\alpha^\wedge( hX), \alpha^\wedge( hY)] - \sigma(\alpha^\wedge( h X)) \alpha^\wedge( hY) + \sigma(\alpha^\wedge( hY)) \alpha^\wedge(  hX) \\
& \quad + \Omega^\wedge( hX, hY) +d\alpha^\wedge( hX,  hY) \\
& =  \Omega^\wedge( hX,  hY) + d\alpha^\wedge( hX,  hY) + [\alpha^\wedge( hX), \alpha^\wedge( hY)] \\
& = \Omega^\wedge( hX,  hY) + (\pr_{\calH}^* d\alpha^\wedge)(hX, hY) + \frac{1}{2} [\alpha^\wedge, \alpha^\wedge](hX,hY).
\end{align*}
\end{linenomath}

Consider the operator $L^\omega \colon \Gamma(T^*M \otimes \Ad P) \to \Gamma(\bigwedge^2 T^*M \otimes \Ad P),$ defined by
\[L^\omega \beta \colon= d^{\nabla^\omega}\! \beta + \frac{1}{2} [\beta, \beta].\]
For $1\leq k\leq r-1$, notice that if $\beta$ vanishes on $D^k$, then $L^\omega\beta(\chi(w)) = - \beta(w)$ for every $w \in D^{k+1}$.
From the above computations, it follows that one has the following equivalence: $\widetilde \omega \in [\omega]_D$ and $\iota_\chi \widetilde \Omega = 0$ if and only if the $\Ad P$-valued one-form $\alpha$ corresponding to $\alpha^\wedge$ satisfies the system of equations
\begin{equation} \label{extAlpha} \alpha|_D = 0, \qquad \iota_\chi L^\omega \alpha = -\iota_\chi \Omega.\end{equation}
This solution exists and is unique according to Corollary~\ref{cor:Covariant}.
\end{enumerate}
\end{proof}

Using Proposition~\ref{th:UsualHol} and its argument, we can now provide the above mentioned versions of Ambrose-Singer's and Ozeki's theorems for equiregular subbundles.
\begin{theorem} \label{th:AS} \emph{(Ambrose-Singer's theorem for horizontal holonomy group)} Let $\pi\colon P \to M$ be a principal $G$-bundle with connection form $\omega$, $D$ a bracket-generating, equiregular subbundle of $P$ of step $r$ and $\chi$ a selector of $D$.  Define the operator $L^\omega \colon \Gamma(T^*M \otimes \Ad P) \to \Gamma(\bigwedge^2 T^*M \otimes \Ad P)$ by
\begin{equation} \label{LOmega} L^\omega \beta \colon= d^{\nabla^\omega}\! \beta + \frac{1}{2} [\beta, \beta].\end{equation}
Let $\Omega$ be the curvature form of $\omega$ and define
\begin{equation} \label{OmegaChi} \Omega_\chi \colon= (\id + L^\omega \iota_\chi)^{r-1} \Omega.\end{equation}
Finally, let $\mathscr{O}_{p_0}$ be the set of points $p \in P$ that can be reached from $p_0$ with $\omega$-horizontal lifts of $D$-horizontal curves. Then the Lie algebra of $\Hol^{\omega,D}(p_0)$ is equal to
\[\{ \Omega_\chi^\wedge(h_p v, h_p w) \, : \, v, w \in T_{\pi(p)}M, p \in \mathscr{O}_{p_0}\}.\]
\end{theorem}
\begin{proof}
Let $\widetilde \omega \in [\omega]_D$ be the unique element with curvature $\widetilde \Omega$ satisfying $\iota_\chi \widetilde \Omega = 0$. We will show that $\Omega_\chi = \widetilde \Omega$. From the proof of Theorem~\ref{th:UsualHol}(b), one has that $\widetilde \Omega = \Omega + L^\omega \alpha$, where $\alpha$ is the unique solution to \eqref{extAlpha}. According to Corollary~\ref{cor:Covariant}, one gets that $\alpha = \iota_\chi \sum_{j=0}^{r-2} \binom{r-1}{j+1} ( L^\omega \iota_\chi )^j \Omega.$
It follows that
\[\widetilde \Omega = \Omega + L^\omega \iota_\chi \sum_{j=0}^{r-2} \binom{r-1}{j+1} ( L^\omega\iota_\chi)^j \Omega = \sum_{j=0}^{r-1} \binom{r-1}{j} (L^\omega \iota_\chi)^j \Omega = \Omega_\chi.\]
We conclude the proof of Theorem~\ref{th:AS} by using Proposition~\ref{th:UsualHol} and the usual Ambrose-Singer Theorem.
\end{proof}

\begin{theorem}\label{th:O}\emph{(Ozeki's theorem for horizontal holonomy group)}
We use the notations of Theorem~\ref{th:AS} and the following ones: let $h$ be the $\omega$-horizontal lift, $p_0$~be an arbitrary point and denote the Lie algebras of $G$ and $\Hol^{\omega,D}(p_0)$ by respectively $\mathfrak{g}$ and $\mathfrak{h}$. For any $p \in P$, define the subspace $\mathfrak{a}(p)$ of $\mathfrak{g}$ by 
$$\mathfrak{a}(p) = \spn \left\{ h X_{1} hX_{2} \dots  hX_{k} \Omega^\wedge_\chi(hY_1, hY_2) \big|_p \, : \,
\begin{array}{c}   Y_1, Y_2 \in \Gamma(TM), \\ X_{j} \in \Gamma(D),  k = 0,1,\dots \end{array} \right\}.
$$
Then $\mathfrak{a}(p_0)$ is a subalgebra of $\mathfrak{h}$. Furthermore,
\begin{enumerate}[\rm (a)]
\item $\mathfrak{h}$ is spanned by $\{ \mathfrak{a}(p) \, \mid \, p \in \mathscr{O}_{p_0} \}.$
\item If $\rank \mathfrak{a}(p)$ is independent of $p$, then $\mathfrak{h} = \mathfrak{a}(p_0)$.
\item If both $\omega$ and $\chi$ are analytic, then $\mathfrak{h} = \mathfrak{a}(p_0)$.
\end{enumerate}
\end{theorem}

\begin{proof}
By Theorem~\ref{th:UsualHol}, let $\widetilde \omega \in [\omega]_D$ be the unique element such that its curvature $\widetilde \Omega$ satisfies $\widetilde \Omega(\chi(\newbullet)) =0$. We then know that $\Hol^{\omega,D}(p) = \Hol^{\widetilde \omega}(p).$ Let $\widetilde h$ denote the $\widetilde \omega$-horizontal lift. We know that $\widetilde \Omega(\widetilde hY_1, \widetilde hY_2) = \Omega_{\chi}(hY_1, hY_2).$ Furthermore, since $L^\omega$ defined in \eqref{LOmega} preserves analyticity, the construction of $\widetilde \omega$ in the proof of Theorem~\ref{th:UsualHol}~(b) gives us that this connection is analytic whenever $\omega$ and $\chi$ are analytic.

Consider the subspaces
$$ \mathfrak{b}(p) = \spn \left\{ \widetilde h Z_{1} \widetilde hZ_{2} \dots   \widetilde hZ_{k} \Omega^\wedge_\chi(hY_1, hY_2) \big|_p \, : \, \\  \begin{array}{c} Y_1, Y_2 \in \Gamma(TM), \\
Z_{j} \in \Gamma(TM),  k = 0,1,\dots \end{array} \right\}.$$
The usual Ozeki theorem along with the above observations means that our desired result holds true with $\mathfrak{b}(p)$ in the place of $\mathfrak{a}(p)$. We will show that $\mathfrak{a}(p) = \mathfrak{b}(p)$ to complete the proof.

If $Z$ is a vector field in $D^{k+1}$ with $\chi(Z) = \sum_{i=1}^l X_i \wedge Y_i$, then
\[\widetilde hZ = \sum_{i=1}^l [\widetilde hX_i, \widetilde hY_i] + \widetilde hZ_2, \qquad X_i, Y_i, Z_2 \in \Gamma(D^k),\]
since $\widetilde \Omega(\chi(\cdot)) = 0$. It follows that we can write $\widetilde hZ$ as a sum of $k$-th order operators constructed with horizontal lifts of elements in $D$, thus yielding
\begin{align*} \mathfrak{b}(p) = \spn \left\{ \widetilde h Z_{1} \widetilde hZ_{2} \dots  \widetilde hZ_{k} \Omega^\wedge_\chi(hY_1, hY_2) \big|_p \, : \, \begin{array}{c}   Y_1, Y_2 \in \Gamma(TM), \\ Z_{j} \in \Gamma(TM),  k = 0,1,\dots \end{array} \right\} \\
= \spn \left\{ \widetilde h X_{1} \widetilde hX_{2} \dots \widetilde hX_{k} \Omega^\wedge_\chi(hY_1, hY_2) \big|_p  \, : \,  \begin{array}{c} Y_1, Y_2 \in \Gamma(TM) \\ X_{j} \in \Gamma(D),  k = 0,1,\dots \end{array} \right\}.\end{align*}
Since $\widetilde \omega \in [\omega]_D$, we have that $\widetilde h X = hX$ for every $X \in \Gamma(D)$ and the result follows.
\end{proof}

\begin{remark}
Let $\omega$ be a connection on the principal $G$ bundle $P \stackrel{\pi}{\to} M$. Proposition~\ref{th:UsualHol} says that each selector $\chi$ gives us a unique connection $\widetilde \omega \in [\omega]_D$ such that $\Hol^{\omega,D}(p) =\Hol^{\widetilde \omega,D}(p) = \Hol^{\widetilde \omega}(p).$ However, we do not claim that these are the only elements satisfying this property. For example, if $\Hol^{\omega,D} = G$, then it trivially follows that $\Hol^{\omega,D} =\Hol^{\widetilde \omega,D} = \Hol^{\widetilde \omega}$ holds for any $\widetilde \omega \in [\omega]_D$.
\end{remark}

\subsection{Horizontal Holonomy of affine connections} \label{sec:Affine}
As in the case of usual holonomy, we can also consider the horizontal holonomy group of an affine connection, as initiated in \cite{HMCK15}. Let $\Pi\colon V \to M$ be a vector bundle with an affine connection~$\nabla$. Let $D$ be a subbundle of $TM$ and use $\mathscr{L}^{D}(x)$ to denote the set of $D$-horizontal loops $\gamma\colon[0,1] \to M$ based at $x \in D$. For $t\in [0,1]$, let $\mathcal{P}_\gamma(t)\colon V_x \to V_{\gamma(t)}$
denote the linear isomorphism defined by the parallel transport along the curve~$\gamma$ in time~$t$. Then we define the horizontal holonomy group of $\nabla$ by
\[\Hol^{\nabla, D}(x) = \left\{ \mathcal{P}_\gamma(1) \in \GL(V_x) \, : \, \gamma \in \mathscr{L}^{D}(x) \, \right\}.\]
Write $\Hol^{\nabla,TM}(x) = \Hol^\nabla(x)$. We say that two connections $\nabla$ and $\widetilde \nabla$ are \emph{$D$-equivalent} if $(\nabla_v - \widetilde \nabla_v)Z = 0$ for any $v \in D$ and $ Z \in \Gamma(V)$. We write $[\nabla]_D$ for the equivalence class of $\nabla$ with respect to this relation. Clearly, $\Hol^{\widetilde \nabla, D}(x) = \Hol^{\nabla,D}(x)$ if $\widetilde \nabla \in [\nabla]_D.$

The correspondence to our theory of principal bundles goes as follows. Let $\nu$ be the rank of $V$ and consider $\mathbb{R}^\nu$ endowed with its canonical basis denoted by $e_1, \dots , e_\nu$. \emph{The frame bundle} $\pi\colon\FGL(V)\to M$ of~$V$ is the principal $\GL(\nu)$-bundle such that for every $x \in M$, the fiber $\FGL(V)_x$ over $x$ consists of all linear isomorphisms $\varphi\colon \mathbb{R}^\nu \to V_x$ and the group $\GL(\nu)$ acts on the right by composition.

From the affine connection $\nabla$, we construct a corresponding principal connection $\omega$ on $\FGL(V)$ as follows. Define $\calH \subseteq T\FGL(V)$ as the collection of all tangent vectors of smooth curves $\varphi$ in $\FGL(V)$ such that, for every $1\leq j\leq \nu$,
$\varphi(\newbullet) e_j$  is a $\nabla$-parallel vector field along $\pi\circ \varphi$. It is standard to check that $\calH \oplus \ker \pi_* = T\FGL(V)$ with $\calH$ being invariant under the group action. Hence, there exists a unique connection form $\omega$ satisfying $\ker \omega = \calH$.

In this case, we can identity $\Ad \FGL(V)$ with the vector bundle $\mathfrak{gl}(V)$ of endomorphisms of $V$ through the mapping $(\varphi, A)/G \mapsto \varphi \circ A \circ \varphi^{-1}$. Furthermore, for any $\varphi \in \FGL(V)_x$, we have the correspondence $\Hol^{\nabla,D}(x) = \varphi \circ \Hol^{\omega,D}(\varphi) \circ \varphi^{-1}$. Also, if $\Omega$ is the curvature form of $\omega$, then $R^{\nabla}(v,w)  = \varphi \circ \Omega^\wedge(h_\varphi v, h_\varphi w) \circ \varphi^{-1}$, where the curvature $R^{\nabla}(X,Y) = [\nabla_X , \nabla_Y] - \nabla_{[X,Y]}$ is seen as a $\mathfrak{gl}(V)$-valued two-form.

We summarize our results so far in this setting.
Let $D$ be a bracket-generating, equiregular subbundle of $TM$ of step $r$ and $\Pi\colon V \to M$ a vector bundle over $M$. If~$\nabla$ is an affine connection on $V$, we will denote the induced connection on $\gl(V)$ by the same symbol. Corresponding to $\nabla$, define an operator $L^{\nabla}\colon \Gamma(T^*M \otimes \mathfrak{gl}(V)) \to \Gamma(\bigwedge^2 T^*M \otimes \gl(V))$ by
\[L^\nabla \alpha = d^\nabla \alpha + \frac{1}{2} [\alpha,\alpha].\]
Then Proposition~\ref{th:UsualHol} and Theorems~\ref{th:AS} and \ref{th:O} read as follows in the case of affine connection.
\begin{theorem} \label{th:affineAS}
Let $\nabla$ be an affine connection on $V$.
\begin{enumerate}[\rm (a)]
\item If there exist a selector $\chi$ of $D$ such that $\iota_\chi R^{\nabla} = 0$, then $\Hol^{\nabla,D}(x) = \Hol^{\nabla}(x)$ for any $x \in M$.
\item For every connection $\nabla$ on $V$ and every selector $\chi$ of $D$, there exists a unique affine connection $\widetilde \nabla \in [\nabla]_D$ such that $\iota_\chi R^{\widetilde \nabla}=0$. The connection $\widetilde \nabla$ is equal to $\nabla + \alpha$ with $\alpha \in \Gamma(T^*M \otimes \mathfrak{gl}(V))$ given by
\[\alpha = \iota_\chi \sum_{j=0}^{r-2} \binom{r-1}{j+1} ( L^\nabla \iota_\chi )^j  R^\nabla.\]
This implies in particular that $\Hol^{\nabla,D}(x)$ is a Lie group.
\item For any $x \in M$, if $\mathfrak{h}$ is the Lie algebra of $\Hol^{\nabla,D}(x)$, then
\[\mathfrak{h} = \left\{\mathcal{P}_\gamma(1)^{-1}R^\nabla_\chi(v,w) \mathcal{P}_\gamma(1) \, : \, \begin{array}{c}\text{$\gamma\colon [0,1] \to M$ is $D$-horizontal} \\ \gamma(0)=x, \gamma(1)=y, v,w \in T_y M \end{array} \right\}. \]
\item For $x \in M$, let $\mathfrak{h}$ denote the Lie algebra of $\Hol^{\nabla,D}(x)$. Let $\chi$ be an arbitrary selector and define
\[R_\chi^\nabla \colon= (\id + L^\nabla \iota_\chi)^{r-1} R^\nabla.\]
For any $y \in M$, define $\mathfrak{a}(y) \subseteq \mathfrak{gl}(V_y)$ given by
\begin{equation} \label{frakax} \mathfrak{a}(y) = \left\{ \nabla_{X_1} \cdots \nabla_{X_k} R^{\nabla}_\chi(Y_1, Y_2) |_y \, : \, \begin{array}{c} Y_1, Y_2 \in \Gamma(TM), \\ X_i \in \Gamma(D) , k=0,1,2 \dots \end{array}  \right\}, \end{equation}
where the symbol $\nabla$ appearing in \eqref{frakax} denotes the connection induced on $\gl(V)$ by $\nabla$. Then $\mathfrak{a}(x)$ is a subalgebra of $\mathfrak{h}$. Furthermore, if $\gamma\colon[0,1] \to M$ is any $D$-horizontal curve with $\gamma(0) = x$ and $\gamma(1) =y$, then $\mathcal{P}_\gamma(1)^{-1} \mathfrak{a}(y)\mathcal{P}_\gamma(1)$ is contained in $\mathfrak{h}$, and $\mathfrak{h}$ is spanned by these subalgebras. Finally, if the rank of $\mathfrak{a}(y)$ is independent of $y$ or if both $\nabla$ and $\chi$ are analytic, then $\mathfrak{a}(x) = \mathfrak{h}.$
\end{enumerate}
\end{theorem}
\begin{remark} Theorem~\ref{th:affineAS} greatly extends the results of \cite{HMCK15} regarding the horizontal holonomy group of an affine connection in case the subbundle $D$ is equiregular and bracket-generating. Note though that it is proved in \cite{HMCK15} that the last conclusion of Item (b), namely that $\Hol^{\omega,D}(x)$ has the structure of a Lie group, still holds true under the sole assumption for $D$ to be completely controllable. This last result has been obtained by recasting horizontal holonomy issues within the framework of development of one manifold onto another one (cf. \cite{CK-MSFM}). 
\end{remark}
\subsection{Equiregular subbundles} \label{sec:NotBG}
The case where the subbundle $D$ is equiregular of step~$r$ but not necessarily bracket-generating can be reduced to the bracket-generating situation described previously by restricting to the leaves of the foliation of $D^r$. According to Frobenius theorem, there exists a corresponding foliation $\calF$ of $M$ tangent to $D^r$. Let $F$ be a leaf of the foliation $\calF$. If $x \in F$, then $D|_x \subseteq D^r|_x = T_xF$. Hence, if $p \in P_x$, then $\Hol^{\omega, D}(p)$ equals $\Hol^{\omega|_F,D|_F}(p)$ where $\omega|_F$ is the restriction of $\omega$ to the principal bundle $P|_F \to F$. By restricting to each leaf of the foliation of $D^r$, we are back to the case where $D$ is also bracket-generating. In particular, if $D$ is integrable, then $D^r = D^1 = D$ and $\Hol^{\omega,D}(p) = \Hol^{\omega|_F}(p)$ for any $p \in P_x, x \in F$.

\subsection{General subbundle} \label{sec:GeneralSB}
Let $D$ be any subbundle of $TM$. Define $\underline{D}^1 = \Gamma(D)$ and, for $k\geq 1$, set
\[
\underline{D}^{k+1} = \left\{ Y, [X,Y] \, : \, Y \in \underline{D}^k, X \in \Gamma(D) \right\}.
\]
For every point $x \in M$, define  \emph{the growth vector} of $D$ at $x \in M$ as the sequence $\underline{n}(x)= (n_k(x))_{k\geq 1}$ where $n_k(x) = \rank \underline{D}^k|_x$.
We say that $x \in M$ is a \emph{regular point} of $D$ if there exists a neighbourhood $U$ of $x$ where the growth vector is constant. We call a point \emph{singular} if it is not regular. Recall that
the set of singular points of $D$ is closed with empty interior, cf.~\cite[Sect.\ 2.1.2, p.\ 21]{Jea14}.

Let $\pi\colon P \to M$ be a principal $G$-bundle with a connection $\omega$. Let $x$ be a regular point of $D$, $p \in P_x$ and $U$ a neighbourhood of $x$ where the growth vector of $D$ is constant. We use $\omega|_U$ to denote the restriction of $\omega$ to $P|_U \to U$. Then one clearly has the following inclusions
\begin{equation} \Hol^{\omega|_U, D|_U}(p) \subseteq \Hol^{\omega,D}(p) \subseteq \Hol^{\omega}(p).\end{equation}
Since $D|_U$ is equiregular, it can be computed with the methods mentioned above.

\begin{example}
Consider $\mathbb{R}^3$ with coordinates $(x,y,z)$. Define $D$ as the span of $\frac{\partial}{\partial x}$ and $\frac{\partial}{\partial y} + x^2 \frac{\partial}{\partial z}$. Then any point $(x,y,z)$ with $x \neq 0$ is a regular point.
\end{example}

\subsection{Comparison with previous results in the contact case}
Assume that $M$ is an oriented $2n+1$-dimensional manifold and that $D$ is an oriented subbundle of rank $2n$. Consider the skew-symmetric tensor
\begin{equation} \label{FirstcalR} \calR\colon \bigwedge^2 D \to TM/D, \qquad \calR(X|_x,Y|_x) = [X,Y]|_x \text{ mod } D,\end{equation}
for any $X, Y \in \Gamma(D).$
As $[X,Y]|_x \text{ mod } D$ does not depend of choices of vector fields~$X$ and~$Y$ extending $X|_x$ and $Y|_x$, this map is a well-defined tensor vanishing if and only if $D$ is integrable. On the other hand, if $\calR$ is surjective and non-degenerate as a bilinear form, then $D$ is called \emph{a contact structure}. In particular, contact structures are bracket-generating and equiregular subbundles of step 2. In Example~\ref{ex:Contact} for instance, the subbundle $D$  is a contact structure.

Horizontal holonomy was first defined for such subbundles $D$ in \cite{FGR97}. On a contact manifold, by Item~(I) in Definition~\ref{def:Selector}, any selector will be a map $\chi \colon TM \to \bigwedge^2 D$ that vanishes on $D$. It follows that $\chi$ can be considered as a map $\chi\colon TM/D \to \bigwedge^2 D$. By Item~(II), the map $\calR$ in \eqref{FirstcalR} is a left inverse to $\chi$. Since both~$M$ and~$D$ are oriented, we get an induced orientation on $TM/D$. Let $\xi$ be a positively oriented basis of $TM/D$ and define $\chi(\xi) = \chi^\xi$. Then $\bigwedge^2 D = \ker \calR \oplus \spn\{ \chi^\xi \}$, and therefore choosing a selector of $D$ is equivalent to choosing a two-vector field~$\chi^\xi$ that never takes values in $\ker \calR$, up to multiplication by non-vanishing smooth functions.
Hence, Proposition~2.4 and Theorem~2.1 in  \cite{FGR97} can be considered as special cases of our results and we have reformulated them below  in terms of $D$-equivalence rather than the language of partial connections used in \cite{FGR97}.
\begin{corollary}
Let $D$ be a contact structure on $M$. Let $P \to M$ be a principal $G$-bundle with a connection form $\omega$. Let $\chi^\xi \in \Gamma(\bigwedge^2 D)$ be a two-vector field that is transverse to $\ker \calR$. Then there exists a unique connection $\widetilde \omega \in [\omega]_D$ whose curvature $\widetilde \Omega$ satisfies $\widetilde \Omega(\chi^\xi) = 0$. Furthermore, if $p_0 \in P$, then the Lie algebra of $\Hol^{\omega,D}(p_0)$ is spanned by $\{  \widetilde \Omega^\wedge(h_p u, h_p v) \, : \, v,w \in T_{\pi(p)} M, p \in \mathscr{O}_{p_0}\}.$
\end{corollary}

\section{Two problems on foliated manifolds} \label{sec:Applications}
\subsection{Totally geodesic foliations} \label{sec:TGF}
Let $M$ be an $n$-dimensional connected manifold and $V\subseteq TM$ an integrable subbundle of rank $\nu \leq n$. By Frobenius Theorem, this subbundle defines us a foliation $\calF$ of $M$ with leaves of dimension $\nu$. Let $D$ be a subbundle of $TM$ such that $TM = D \oplus V$. The issue we address below regards the existence of a Riemannian metric~$\tensorg$ on $M$ such that $D$ is the $\tensorg$-orthogonal complement to $V$ and $\calF$  is totally geodesic, i.e., the leaves of $\calF$ are totally geodesic submanifolds of $M$.

Let $\tensorg$ be a Riemannian metric on $M$ with corresponding Levi-Civita connection $\nabla^{\tensorg}$. A submanifold $F$ on $M$ is called \emph{totally geodesic} if the geodesics of $(F, \tensorg|_{TF})$ are also geodesics in $(M, \tensorg)$. Consider an integrable subbundle $V$ of $TM$ and denote by~$D$ its $\tensorg$-orthogonal complement. Let
$\pr_V$ and $\pr_D$ be the $\tensorg$-orthogonal projections onto~$V$ and~$D$, respectively.
The foliation $\calF$ of $V$ is totally geodesic if and only if
\[\II(Z_1, Z_2) := \pr_D \nabla_{\pr_V Z_1}^{\tensorg} \pr_V Z_2 =0, \qquad\text{for every } Z_1, Z_2 \in \Gamma(TM).\]
The vector-valued tensor $\II$ is symmetric and is called \emph{the second fundamental form} of the leaves of~$\calF$. Furthermore, for any vector field $X \in \Gamma(D)$ and $Z_1, Z_2 \in \Gamma(V)$, one must have
\[0=-2 \tensorg(X, \nabla_{Z_1}^{\tensorg} Z_2) = -2 \tensorg(X, \II(Z_1, Z_2)) = (\calL_X \pr_V^* \tensorg)(Z_1,Z_2).\]
Hence, the foliation $\calF$ is totally geodesic if and only if $ \calL_X \pr_V^* \tensorg = 0$ for every vector field $X$ with values in $D$.

Choose any affine connection $\nabla$ on $V$ such that for any $X \in \Gamma(D)$ and $Z \in \Gamma(V)$, we have
\begin{equation} \label{verticalC} \nabla_X Z = \pr_V [X,Z].\end{equation}
Such a connection is called \emph{vertical} and all such vertical connections are $D$-equivalent. Write $\tensorg = \pr_D^* \tensorg^{D} + \pr_V^* \tensorg^{V},$ with metric tensors $\tensorg^D$ and $\tensorg^V$ on $D$ and $V$ respectively. Then $\calF$ is totally geodesic if and only if
\begin{equation} \label{TGparallel} \nabla_v\tensorg^V =0, \qquad \text{for any } v \in D.\end{equation}
This gives us the following reformulation of our problem. Given the direct sum $TM = D \oplus V$ of subbundles such that $V$ is integrable, under what conditions does there exist a metric tensor $\tensorg^V$ on $V$ which is parallel in the directions of $D$ with respect to a connection $\nabla$ satisfying \eqref{verticalC}?

For any connection $\nabla$ on $V$ satisfying \eqref{verticalC}, denote the induced connection on $\Sym^2 V^*$ by the same symbol. For $x\in M$, let $\Hol^{\nabla,D}(x) \subseteq \GL(\Sym^2 V^*_x)$ be the holonomy group with respect to $\nabla$ at $x$. It is clear then that  there exists $\tensorg^V\in\Gamma(\Sym^2 V^*)$ satisfying~\eqref{TGparallel} if and only if
\[
a \cdot \tensorg^V|_x = \tensorg^V|_x, \qquad \text{for any } a \in \Hol^{\nabla,D}(x).
\]
Hence, if $\calF$ is totally geodesic then $\Hol^{\nabla,D}(x)$ admits a positive definite fixed point for any $x \in M$. 
As a consequence, if there exists a metric $\tensorg^V$ satisfying \eqref{TGparallel}, then there exists a positive definite element in $\Sym^2 V_x^*$ belonging to the kernel of every element in the Lie algebra $\mathfrak{h}$ of $\Hol^{\nabla,D}(x)$.

Assume moreover that $D$ is completely controllable. Then for any $x \in M$, there is a one-to-one correspondence between positive definite fixed points of $\Hol^{\nabla,D}(x)$ and all metrics $\tensorg^V \in \Gamma(\Sym^2 V^*)$ satisfying \eqref{TGparallel}, since any such fixed point can be extended uniquely to a metric tensor through parallel transport. If $\Hol^{\nabla,D}(x)$ is connected, then there is a one-to-one correspondence between such metrics and positive definite points in $\cap_{A \in \mathfrak{h}} \ker A$. In particular for the case when $D$ is bracket-generating and equiregular, let $R^{\nabla}_\chi $ be defined with respect to some selector $\chi$ of $D$ as in Section~\ref{sec:Affine}. Then we must have $R^{\nabla}_\chi(\newbullet, \newbullet) \tensorg^V =0$ for any metric $\tensorg^V$ satisfying \eqref{TGparallel}. Recall by Remark~\ref{re:connected} that $\Hol^{\nabla,D}$ is connected whenever $M$ is simply connected and $D$ is bracket-generating. Hence, if the latter two conditions are satisfied for respectively $M$ and $D$ and $R^\nabla_\chi =0$, then there is a one-to-one correspondence between positive definite elements in $\Sym^2 V^*$ and metrics satisfying \eqref{TGparallel}. We note that if $\mathbf{e} \in \Sym^2 V^*_x$, then
\[(R^\nabla_{\chi}(v_1,v_2) \mathbf{e})(w_1,w_2) = - \mathbf{e}(R^\nabla_\chi(v_1,v_2) w_1, w_2)  - \mathbf{e}( w_1, R^\nabla_\chi(v_1,v_2) w_2), \]
for any $v_j,w_j \in T_xM$, $j=1,2$, where the $R^\nabla_\chi$ on the right hand side is defined relative to $\nabla$ seen as a connection on $V$.

We summarize the findings of this section in the following theorem.
\begin{theorem}\label{th:TGF}
Let $V \subseteq TM$ be an integrable subbundle of $M$ and let $D$ be any subbundle such that $TM = D \oplus V$. Let $\nabla$ be any connection on $V$ satisfying \eqref{verticalC}.

Assume that there exists a Riemannian metric $\tensorg$ such that
\begin{equation} \label{TGg}
D\text{ is the $\tensorg$-orthogonal of $V$ and the foliation of $V$ is totally geodesic.}
\end{equation}
Then, for every $x \in M$, the group $\Hol^{\nabla, D}(x) \subseteq \GL(\Sym^2 V^*_x),$
has a positive definite fixed point.

Furthermore, assume that $D$ is completely controllable. Then $M$ has a Riemannian metric $\tensorg$ satisfying \eqref{TGg} if and only if there exists a point $x\in M$ such that $\Hol^{\nabla,D}(x) \subseteq \GL(\Sym^2 V^*_x)$ has a positive definite fixed point.
\end{theorem}
We illustrate the difference between the case when $D$ is integrable and the one when $D$ is completely controllable by studying a specific framework. Let $V$ be an integrable subbundle of $TM$ of rank $\nu$. Assume that $M$ is equipped with a Riemannian metric $\tensorg$ and let $D$ be the orthogonal complement of $V$. Following~\cite{Rei59}, we say that $\calF$ is \emph{a Riemannian foliation} if
\begin{equation} \label{RiemFol} \calL_{Z} \pr_{D}^* \tensorg =0, \qquad \text{for any } Z \in \Gamma(V).\end{equation}
Riemannian foliations locally look like a \emph{Riemannian submersion}. Recall that a surjective submersion between two Riemannian manifolds $f\colon (M, \tensorg) \to (B, \tensorg^B)$ is called Riemannian if $\tensorg(v, w) = \tensorg^B(\pi_* v, \pi_* w)$ for any $v,w \in (\ker f_*)^\perp.$ If $\calF$ is a Riemannian foliation of $(M, \tensorg)$, any point $x \in M$ has a neighbourhood $U$ such that $B = U/\calF|_U$ is a well-defined manifold that can be given a metric $\tensorg^B$, making the quotient map $f\colon U \to B$ into a Riemannian submersion.

Let $M = B \times F$ be the product of two connected manifolds, $\calF$ be the foliation with leaves $\{ (b, F) \, : \, b \in B\}$, $V$ be the corresponding integrable subbundle, and $D$ be any subbundle such that $TM = D \oplus V$. Note that if $f\colon M \to B$ is the projection, then $f_*|_D\colon D \to TB$ is bijective on every fiber, meaning that curves in $B$ have $D$-horizontal lifts, at least for short time. We want to know if there exists a Riemannian metric $\tensorg = \pr_D^* \tensorg^D + \pr_V^* \tensorg^V$ such that the leaves of $\calF$ are totally geodesic. Since this question does not depend on $\tensorg^D$, we may choose $\tensorg^D = f^* \tensorg^B|_D$ for some Riemannian metric on $B$. This will make $\calF$ into a Riemannian foliation.

If $D = f^*TB$, we can choose any metric $\tensorg^V$ on $V$ and the foliation is totally geodesic. On the other hand, if $D$ is completely controllable, the choice of metric is more restrictive. Indeed, if $\calF$ is totally geodesic and $\tensorg$ is a complete metric, then the leaves of $\calF$ are homogeneous spaces. For any smooth curve $\gamma\colon [0,1] \to B$ consider the map $\phi_\gamma\colon \gamma(0) \times F \to \gamma(1) \times F$ sending a point $(\gamma(0), z_0)$ to the point $(\gamma(1), z_1)$ if the latter is the endpoint of the horizontal lift of $\gamma$ starting at $(\gamma(0),z_0)$. By \cite[Proposition 3.3]{Her60}, these maps are isometries. If $D$ is completely controllable, then the isometry group of $(b,F)$ must act transitively for any $b \in B.$

\begin{remark}
\begin{enumerate}[\rm (a)]
\item We observe that by reversing the role of $D$ and $V$, we may consider the following question. Given an integrable subbundle $V$ and complement $D$, when does there exist a Riemannian metric $\tensorg$ on $M$ such that $V$ and $D$ are orthogonal and such that the foliation $\calF$ of $V$ is Riemannian? Let $\nabla^\prime$ be any affine connection on $D$ that satisfies
\[\nabla_Z^\prime X = \pr_V [Z,X], \qquad \text{for any } Z \in \Gamma(V), Z \in \Gamma(D).\]
Then \eqref{RiemFol} can be reformulated as $\nabla_v^\prime \tensorg^D = 0$ for any $v \in V$ where $\tensorg^D = \tensorg|_D$. Hence, the foliation $\calF$ can only be made Riemannian with transversal subbundle $D$ if $\Hol^{\nabla^\prime,V}(x)$ has a positive definite fixed point for any $x \in M$.
\item In general, if $f\colon M \to B$ is a surjective submersion, $\calF = \{ f^{-1}(b) \, : \, b \in B\}$ is a Riemannian totally geodesic foliation and $\tensorg$ is a complete metric, then by~\cite{Her60} we know that the leaves of the foliation are isometric to some manifold~$F$. Furthermore, if $G$ is the isometry group of $F$, then there exist a principal $G$-bundle $P\to M$ (with action written on the left) such that
\begin{equation} \label{MPF} M = G\setminus(P \times F),\end{equation}
The quotient is here with respect to the diagonal action.
\item Even if $TM = D \oplus V$ with $V$ not integrable, we can still define a connection $\nabla$ as in \eqref{verticalC}. Studying when there exists a metric parallel with respect to $\nabla$ along $D$-horizontal curves still has applications to the heat flow of subelliptic operators when $V$ is not integrable, see \cite[Section~3.3 \& 3.8]{GrTh15a} and \cite[Section~6.5]{GrTh15b}. Theorem~\ref{th:TGF} is also applicable to this case.
\end{enumerate}
\end{remark}

\subsection{Submersions and principal bundles} \label{sec:PB}
Let $F \to M \stackrel{f}{\to} B$ be a fiber bundle over a manifold $B$, with the fiber $F$ being a connected manifold. Let $V = \ker f_*$ be of rank $\nu$ and let~$D$ be an Ehresmann connection on~$f$, i.e., a subbundle such that $TM = D \oplus V$. The foliation of $V$ is given by $\{ M_b := f^{-1}(b) \, \colon \, b\in B\}.$ We ask the following question: under what conditions does there exist a group action on $M$ rendering $f$ a principal bundle and $D$ a principal connection? In order to approach this question, we first look at its infinitesimal version, namely, when does there exist vector fields $Z_1, \dots, Z_\nu$ on $M$ satisfying
\begin{enumerate}[\rm (i)]
\item $V = \spn \{ Z_1, \dots, Z_\nu\}$;
\item for any $X \in \Gamma(D)$ and any $i =1, \dots, \nu$, $[Z_i, X]$ has values in $D$;
\item $\mathfrak{g} = \spn \left\{Z_i \, \colon \, i=1, \dots, \nu \right\}$ is a subalgebra of $\Gamma(TM)$.
\end{enumerate}
If there exists a group $G$ acting on the fibers of $f:M \to B$ such that both $f$ and $D$ are principal, we can obtain the desired vector fields above by defining $Z_i|_x = \sigma(A_i)$ for some basis $A_1, \dots, A_\nu$ of the Lie algebra of $G$. The map $\sigma$ is here defined as in \eqref{SigmaVF}. Conversely, let $Z_1, \dots, Z_\nu$ be vector fields satisfying Item (i), (ii) and (iii). If these vector fields are also complete and if $G$ is the simply connected Lie group of $\mathfrak{g}$, we get a group action $G \times M \to M$. By the definition of~$\mathfrak{g}$, this group action preserves the fibers of $f$ and is \emph{locally free}, i.e. the stabilizers $G_x = \{ a \in G \, \colon \, a \cdot x = x \}$ are discrete groups for any $x \in M$. Because of this, for any $x \in M$ with $f(x) = b$, the map $G \to M_b$, $a \mapsto a \cdot x$, is a surjective local diffeomorphism. This map must be a diffeomorphism if $M_b$ is simply connected, that is, if $F$ is simply connected. Hence, if $F$ is simply connected, the group $G$ acts freely and transitively on every fiber and $f$ is a principal bundle with principal connection $D$.

Define $S_0(D)$ be the space of \emph{infinitesimal vertical symmetries}
i.e. the space of all $Z\in\Gamma(TM)$ such that $f_*Z=0$ and such that $[Z,X]$ has values in $D$ for all $X\in\Gamma(D)$. This is a sub-algebra of $\Gamma(TM)$ by the Jacobi identity. If~$\mathfrak{g}$ is as in Item~(iii), it will be an $\nu$-dimensional subalgebra of $S_0(D)$. Note that if $\nabla$ is any connection on $V$ satisfying \eqref{verticalC}, then $\nabla_v Z = 0$ for  any $Z \in S_0(D)$ and $v \in D$. In particular, this must hold for the vector fields $Z_1, \dots, Z_\nu$. If $D$ is completely controllable, $S_0(D)$ must be of dimension $\leq \nu$ since every element $Z$ is uniquely determined by its value at one point through parallel transport. Hence, it follows that the requirement of Item (iii) is superfluous, since if there exists vector fields $Z_1, \dots, Z_\nu$ satisfying Item (i) and (ii), then $\mathfrak{g} = \spn \{ Z_1, \dots, Z_\nu\} = S_0(D)$. 

We conclude the following from the above discussion.

\begin{theorem}\label{th:SUB}
Let $F \to M \stackrel{f}{\to} B$ be a fiber bundle with connected fiber $F$, $V : = \ker f_*$ the vertical bundle and $D$ an Ehresmann connection on $f$. Let $\nabla$ be any connection on $V$ satisfying \eqref{verticalC}.

If there exists a group action of some Lie group $G$ on $M$ such that $f$ becomes a principal bundle and $D$ becomes a principal Ehresmann connection, then, for every $x\in M$, one has $\Hol^{\nabla,D}(x) = \{\mathrm{Id}\} \subseteq \GL(V_x)$, where $\mathrm{Id}$ denotes the identity mapping.

Furthermore, assume that $D$ is completely controllable and that $F$ is compact and simply connected. Then $f$ can be given the structure of a principal bundle with principal connection $D$ if and only if there exists a point $ x\in M$ such that $\Hol^{\nabla,D}(x) = \{\mathrm{Id}\} \subseteq \GL(V_x)$.
\end{theorem}

For a related result in a special case, cf. \cite{CGK15}.

Assume that $D$ is bracket generating and equiregular with a selector~$\chi$. Then $R^{\nabla}_\chi \equiv 0$ is a necessary condition for $D$ to be a principal connection. If $F$ is compact and simply connected and $M$ is simply connected as well, then the condition $R^{\nabla}_\chi\equiv 0$ condition is also sufficient. We emphasize that the only reason for the assumption that $F$ is compact, is to ensure that the vector fields $Z_1, \dots, Z_\nu$ mentioned above are complete.

\begin{remark}
\label{rem:srgeom}
Theorem~\ref{th:SUB} has applications to sub-Riemannian geometry since the Lie algebra $S_0(D)$ often appear as a subalgebra of the Lie algebra of infinitesimal isometries. Recall first that an isometry $\phi$ of a sub-Riemannian manifold $(M, D, \tensorh)$ is a diffeomorphism satisfying $\phi_* D \subseteq D$ with $\tensorh(\phi_* v, \phi_* w) = \tensorh( v,  w)$ for any $v,w \in D$. An infinitesimal isometry is a vector field $Z$ such that, for every $X \in \Gamma(D)$, the vector field $[Z,X]$ takes values in $D$ and $Z \tensorh(X,X) = 2 \tensorh(X, [Z,X])$. Let $F \to M \stackrel{f}{\to} B$ be a fiber bundle over a Riemannian manifold $(B, \tensorg^B)$ and $D$ be an Ehresmann connection on $M$. Define a metric on $D$ by $\tensorh = f^* \tensorg^B |_D$ and set $V = \ker f_*$ for the vertical bundle with corresponding foliation $\calF$. Let $\tensorg$ be any metric on $M$ such that $\tensorg|_D = \tensorh$ and $V$ is the $\tensorg$-orthogonal complement of $D$. By definition, $f$ is a Riemannian submersion, $\calF$ a Riemannian foliation and hence, for every $Z \in \Gamma(V)$,
\[Z \tensorh(X, X) = 2 \tensorh(X, \pr_D [Z,X]) \quad \text{ for any } X \in \Gamma(D).\]
In order for $Z$ to be an infinitesimal sub-Riemannian isometry, we must also have $\pr_V [X,Z] = 0$ for any $X \in \Gamma(D)$. In other words, if~$\nabla$ is a connection on~$V$ satisfying \eqref{verticalC}, then $Z$ must be parallel along $D$-horizontal curves. If $\Hol^{\nabla,D}(x) = \{ \mathrm{Id} \} \subseteq \GL(V_x)$, then $V$ has a basis of infinitesimal isometries.
\end{remark}

\subsection{Computation of the curvature}
\label{sec:curvature}
The above problems involve the use of connections corresponding to an affine connection $\nabla$ satisfying \eqref{verticalC}. In order to apply the results of Section~\ref{sec:Affine}, one must choose a selector and appropriate affine connections $\nabla$. We next give two options for the choice of $\nabla$ and compute their curvatures. We end with a remark on how to compute $R^\nabla_\chi$ from $R^\nabla$.

\subsubsection{Connection appearing from the choice of an auxiliary metric}
Choose an auxiliary metric $\widetilde \tensorg$ on $M$ with $D$ and $V$ orthogonal and define
\begin{equation} \label{nablaV} \nabla_Y Z := \pr_V [\pr_D Y, Z] + \pr_V \nabla_{\pr_V Y}^{\widetilde \tensorg} Z, \qquad Y \in \Gamma(TM), Z \in \Gamma(V).\end{equation}
To compute the curvature of this connection, we consider the \emph{curvature of} $D$ with respect to the complement~$V$, i.e, the vector-valued two-form $\mathcal{R}$ on $M$ defined by
\[\mathcal{R}(X,Y) = \pr_V [\pr_D X, \pr_D Y], \qquad X, Y \in \Gamma(TM).\]
It is a well-defined two-form since it is skew-symmetric and $C^\infty(M)$-linear in both arguments. We extend the connection $\nabla$ to a connection $\rnabla$ on $M$ by setting
\begin{linenomath}
\begin{align} \label{rnabla} \rnabla_Y X & :=  \pr_V [\pr_D Y, \pr_V X] + \pr_V \nabla_{\pr_V Y}^{\widetilde \tensorg} \pr_V X \\ \nonumber
& \quad +  \pr_D [\pr_V Y, \pr_D X] + \pr_D \nabla_{\pr_D Y}^{\widetilde \tensorg} \pr_D X. \end{align}
\end{linenomath}
\begin{proposition} \label{prop:Rnabla}
Endow $M$ with some Riemannian metric $\widetilde \tensorg$ and consider the connection $\nabla$ on $V$ given by \eqref{nablaV}. Let $\II$ be the second fundamental form with respect to $\widetilde \tensorg$. For any $X,Y \in \Gamma(TM)$ and $Z \in \Gamma(V)$, the curvature of $\nabla$ is given by
\[
 R^{\nabla}(X,Y)= R^{\calF}(\pr_V X, \pr_V Y)-(\rnabla_{\centerdot} \calR)(X,Y)
 + \mathscr{S}(X,Y) - \mathscr{S}(Y,X),
 \]
where $R^{\calF}$ is the curvature of the leaves $F$ of the foliation $\calF$ with respect to
$\widetilde \tensorg|_F$ and $\mathscr{S}(X,Y)$ is the unique endomorphism satisfying the following: for every $Z_1,Z_2\in \Gamma(TM)$,
\[\widetilde \tensorg(\mathscr{S}(X,Y) Z_1, Z_2) = \widetilde \tensorg(X, (\rnabla_{Z_2} \II)(Y,Z_1) - (\rnabla_{\pr_V Y} \II)(Z_1,Z_2) - (\rnabla_{Z_1} \II)(Y, Z_2)).\]
\end{proposition}

\begin{proof}
Since $\rnabla$ is a direct sum of a connection on $D$ and $V$, we have $R^{\nabla}(X,Y) Z = R^{\rnabla}(X,Y) Z$ for any $X,Y \in \Gamma(TM)$, $V \in \Gamma(V)$. Hence, we need to compute $R^{\rnabla}(X,Y) Z$ for $Z$ taking values in~$V$.

Let us first consider the case when $X$ and $Y$ both take values in $V$. Let $F$ be a leaf of $\calF$. Observe that $\rnabla_Y Z$ is equal to $\nabla^{\widetilde \tensorg|_F}_{X|_F} Z|_F $ on $F$, the latter connection being the Levi-Civita connection of $\widetilde \tensorg|_F$. Hence $R^{\rnabla}(X,Y)Z = R^{\calF}(X,Y)Z$.

If both $X$ and $Y$ takes values in $D$, then by the first Bianchi identity
\begin{linenomath}
\begin{align*}
& R^{\rnabla}(X,Y) Z = \pr_V R^{\rnabla}(X,Y) Z = \pr_V \circlearrowright R^{\rnabla}(X,Y) Z \\
& = \pr_V \circlearrowright T^{\rnabla}( T^{\rnabla}(X,Y), Z) + \pr_V \circlearrowright  (\rnabla_X T^{\rnabla})(Y,Z),
\end{align*}
\end{linenomath}
with $\circlearrowright$ denoting the cyclic sum and $T^{\rnabla}$ denoting the torsion tensor of $\rnabla$. From the definition of $\rnabla$, it is simple to verify that $T^{\rnabla} = - \calR$. Since $\calR$ vanishes on $V$, we obtain
\[R^{\rnabla}(X,Y) Z = - (\rnabla_Z \calR)(X,Y).\]

For the last part, we consider the case when $X$ takes values in $D$ and $Y$ in $V$ respectively. Observe again from the first Bianchi identity, that one has
\[0 = \pr_V \circlearrowright R^{\rnabla}(X,Y) Z = R^{\rnabla}(X,Y)Z - R^{\rnabla}(X,Z)Y.\]
Hence, we only need to compute $R^{\rnabla}(X,Y)Y$ to derive the result. From the definition of $\rnabla$, observe that for any vector field $Z$ with values in $V$, we have
\[(\rnabla_X \widetilde \tensorg)(Z,Z) = - 2 \widetilde \tensorg(X,\II(Z,Z)), \qquad (\rnabla_Y \widetilde \tensorg)(Z,Z) = 0.\]
It follows that for any vector field $Z_1, Z_2$ taking values in $V$, we have
\begin{eqnarray*}
\widetilde\tensorg(R(X,Y)Z_1, Z_2) & =& - \widetilde \tensorg(R^{\rnabla}(X,Y) Z_2, Z_1) - (R^{\rnabla}(X,Y) \widetilde \tensorg)(Z_1, Z_2), \\
(R^{\rnabla}(X,Y)\widetilde \tensorg)(Z_1, Z_2) & = & 2 \widetilde \tensorg(X, (\rnabla_Y \II)(Z_1, Z_2)).
\end{eqnarray*}
This leads to the conclusion that
\begin{linenomath}
\begin{align*}
& \widetilde \tensorg(R^{\rnabla}(X,Y)Y,Z) = - \widetilde \tensorg( R^{\rnabla}(X,Y) Z, Y) - 2 \widetilde \tensorg(X, (\rnabla_Y \II)(Y,Z)) \\
& = - \widetilde \tensorg( R^{\rnabla} (X,Z) Y, Y) - 2 \widetilde \tensorg(X, (\rnabla_Y \II)(Y,Z)) \\
& = \widetilde \tensorg(X , (\rnabla_Z \II)( Y, Y)) - 2 \widetilde \tensorg(X, (\rnabla_Y \II)(Y,Z)).
\end{align*}
\end{linenomath}
As a result, we have
\[R^{\rnabla}(X,Y) Z = \sharp \widetilde \tensorg(X, (\rnabla_{\centerdot} \II)(Y,Z) - (\rnabla_{Y} \II)(Z,\newbullet) - (\rnabla_Z \II)(Y, \newbullet)) = \mathscr{S}(X,Y)Z.\]
\end{proof}

\subsubsection{Connection appearing from a global basis}
Assume that $V$ admits a global basis of vector fields $Z_1, \dots Z_\nu$. Let $\tau_1, \dots, \tau_\nu$ be the corresponding dual one-forms, i.e., they vanish on $D$ and satisfy $\tau_i(Z_j) = \delta_{ij}$ for every $1\leq i,j\leq \nu$. Define the connection $\nabla$ by
\[\nabla_X Z_j := \pr_V [\pr_D X,Z_j],\]
for any $j =1, \dots, \nu$. It is then well-defined for any $Z \in \Gamma(V)$ through the Leibniz property, i.e.,
\begin{equation} \label{Zconnection} \nabla_X Z = \pr_V [\pr_D X, Z] +
\sum_{j=1}^\nu ((\pr_V X)\tau_i(Z)) Z_i.
\end{equation}

\begin{proposition} \label{prop:FlatNabla}
Define a $\gl(V)$-valued one-form $\alpha$ by
\begin{equation} \label{AlphaEndomorphism} \alpha(X) Z_j = \sum_{i=1}^n \alpha_{ij}(X) Z_i, \qquad \alpha_{ij}= \pr_D\calL_{Z_i} \tau_j.
\end{equation}
The curvature $R^\nabla$ of the connection $\nabla$ defined in \eqref{Zconnection} is given by
\[R^\nabla = L^\nabla \alpha - [\alpha, \alpha].\]
\end{proposition}
We recall that $[\alpha, \alpha](X,Y) = 2 \alpha(X) \alpha(Y) - 2\alpha(Y)\alpha(X).$
\begin{proof}
Similarly to the proof of Theorem~\ref{th:UsualHol}~(b), we can show that for any pair of connections $\nabla$ and $\widetilde \nabla$ such that $\nabla = \widetilde \nabla + \alpha$, $\alpha \in \Gamma(T^*M \otimes \gl(V))$, one has
\[R^{\nabla} - R^{\widetilde \nabla} = L^{\widetilde \nabla} \alpha, \qquad (L^{\nabla} - L^{\widetilde \nabla}) \alpha = (d^{\nabla} - d^{\widetilde \nabla}) \alpha = [\alpha, \alpha],\]
so $R^{\nabla} - R^{\widetilde \nabla} = L^{\nabla} \alpha - [\alpha, \alpha]$. Choose $\widetilde \nabla$ as the connection on $V$ determined by $\widetilde \nabla Z_i = 0$ and $\alpha$ as in \eqref{AlphaEndomorphism} to get the result.
\end{proof}

\subsubsection{Computing $R^\nabla_\chi$}

Assume that~$D$ is bracket-generating and equiregular of step~$r$. Define $\eta^0 = R^{\nabla}$ as a $\mathfrak{gl}(V)$-valued two-form and for $k=0, \dots, r-2$, define $\eta^{k+1} = L^\nabla \iota_\chi \eta^k$ for some selector $\chi$ of $D$. Let $\rnabla$ be any connection extending $\nabla$, with torsion $T^{\rnabla}$. We then have
\begin{linenomath}
\begin{align} \label{EtaK} \eta^{k+1}(X,Y) & = (\rnabla_X \iota_\chi \eta^k)(Y) - (\rnabla_Y \iota_\chi \eta^k)(X) + \eta^k(\chi(T^{\rnabla}(X,Y))) \\ \nonumber
& + \eta^k(\chi(X)) \eta^k(\chi(Y))  - \eta^k(\chi(Y)) \eta^k(\chi(X)) , \end{align}
\end{linenomath}
where $(\rnabla_X \iota_\chi \eta^k)(Y)Z = \rnabla_X \eta^k(\chi(Y))Z - \eta^k(\chi(\rnabla_X Y))Z - \eta^k(\chi(Y)) \rnabla_X Z$. It follows that if $R^\nabla_\chi$ is defined as in Section~\ref{sec:Affine},
then
\[R^{\nabla}_{\chi} = \sum_{j=0}^{r-1} \binom{r-1}{j} \eta^j.\]

\section{Examples}\label{sec:examples}
\subsection{One-dimensional foliations}
Let $M$ be a connected, simply connected manifold with a foliation $\calF$ corresponding to an integrable one-dimensional subbundle $V$. Let $D$ be any subbundle that is transverse to $V$. We assume that $V$ is orientable, i.e., $V = \spn \{ Z\}$ for some vector field $Z$. It follows that $\Ann(D)$, the subbundle of $T^*M$ consisting of covectors vanishing on $D$, is oriented as well, and we write $\Ann(D) = \spn\{ \tau\}$.

\begin{proposition} \label{prop:Zfoliation}
Let $\tau$ be a non-vanishing one-form on a connected, simply connected manifold $M$. Define $D = \ker \tau$ and assume that $d\tau |_{\bigwedge^2 D}$ is non-vanishing at any point. Let $Z \in \Gamma(TM)$ and $\chi \in \Gamma(\bigwedge^2 D)$ be respectively a vector field and a two-vector field such that
\[\tau(Z) = 1, \qquad d\tau(\chi) = -1.\]
Consider the operator $d_{ \tau \otimes \chi} = d (\id + \iota_{\chi \otimes \tau} d)$ on one-forms.
Then there exists a Riemannian metric $\tensorg$ on $M$ such that the foliation tangent to $Z$ is totally geodesic and $D$ is orthogonal to $V = \mathrm{span} \,\{ Z\}$ if and only if
\begin{equation} \label{1dimdchi} d_{\tau \otimes \chi} \calL_Z \tau = 0.\end{equation}
\end{proposition}
Recall from Lemma~\ref{lemma:Selector}~(d) that \eqref{1dimdchi} is indeed independent of choice of $\chi$.
\begin{proof}
Let $\pr_D$ and $\pr_V$ be the respective projections to $D$ and $V$ relative to the direct sum $TM = D \oplus V$. Since $\pr_D^* d\tau$ never vanishes, $D$ must be bracket-generating and equiregular of step~2. Furthermore, if $\chi$ is any two-vector field such that $d\tau(\chi) = -1$, then $\tau \otimes \chi$ is a selector of $D$.
Define a metric $\nabla$ on $V$ by $\nabla_X Z = \pr_V [\pr_D X,Z] = (\calL_Z \tau)(\pr_D X) Z$ for any $X \in \Gamma(TM)$. Define $\alpha \in \Gamma(T^*M \otimes \gl(V))$ as $\alpha(X) v = (\calL_Z \tau)(\pr_DX) v = (\calL_Z \tau)(X) v$ for any $v \in V$. We use Proposition~\ref{prop:FlatNabla} to obtain that
\begin{linenomath}
\begin{align*}
R^{\nabla}_{\tau \otimes \chi} & = (\id + L^\nabla \iota_{\tau \otimes \chi}) (L^\nabla \alpha - [\alpha, \alpha]) = (\id + d^{\nabla} \iota_{\tau \otimes \chi}) d^\nabla \alpha \\
&  = d^{\nabla}( \id + \iota_{\tau \otimes \chi} d^\nabla) \alpha = d^\nabla(\id + \iota_{\tau \otimes \chi} d^\nabla) \calL_Z \tau \otimes \id_V \\
& = \left( d(\id + \iota_{\tau \otimes \chi} d) \calL_Z \tau \right) \otimes \id_V = d_{\tau \otimes \chi} \calL_Z \tau \otimes \id_V.
\end{align*}
\end{linenomath}
Since $V$ is one-dimensional, $R^\nabla_\chi = 0$ is a necessary condition for the existence of a metric $\tensorg$ such that $D$ and $V$ are orthogonal and the foliation of $V$ is totally geodesic. Since $M$ is simply connected and $D$ is bracket-generating, it is also a sufficient condition.
\end{proof}

\begin{example}
Consider $\mathbb{R}^3$ with coordinates $(x,y,z)$. Consider the following global basis of the tangent bundle,
\begin{eqnarray*}
X & = & \frac{\partial}{\partial x} - \frac{1}{2} y \frac{\partial}{\partial z} \, , \\
Y & = & \frac{\partial}{\partial y} + \frac{1}{2} x \frac{\partial}{\partial z} \, , \\
Z & = & \frac{\partial}{\partial z} + \phi_1 X + \phi_2 Y \, ,
\end{eqnarray*}
where $\phi_1$ and $\phi_2$ are two smooth arbitrary functions on $\mathbb{R}^3$. Define $D$ to be the span of $X$ and $Y$ and let $V = V(\phi_1,\phi_2)$ be the span of $Z$. For which functions $\phi_1$ and $\phi_2$ does there exist a Riemannian metric $\tensorg$ such that the foliation $\calF$ of $V$ is totally geodesic and $D$ is the $\tensorg$-orthogonal complement of $V$?

We define $\tau$ such that $\tau(X) = 0$, $\tau(Y) = 0$ and $\tau(Z) = \tau( \frac{\partial}{\partial z}) = 1$. Then
\[\tau = dz + \frac{1}{2} y dx - \frac{1}{2} x dy.\]
The unique element $\chi$ in $\bigwedge^2 D$ satisfying $d \tau (\chi) = -1$ is $\chi = X \wedge Y.$ Its Lie derivative with respect to $Z$ is given by
\[\calL_{Z} \tau =  - \phi_1 dy + \phi_2 dx.\]
Furthermore,
\begin{linenomath}
\begin{align*}
d_{\chi \otimes \tau} \calL_Z \tau &= d( \calL_Z \tau + (d \calL_Z \tau)(X,Y) \tau) = d( - \phi dy + \phi_2 dx - (X \phi_1 + Y \phi_2) \tau) \\
&= - d\phi_1 \wedge dy + d\phi_2 \wedge dx  - d(X \phi_1 + Y \phi_2) \wedge \tau - (X\phi_1 + Y\phi_2) d\tau \\
& = - d(X\phi_1 + Y\phi_2) \wedge \tau.
\end{align*}
\end{linenomath}
Hence $d_{\tau \otimes \chi} \calL_Z \tau = 0$ if and only if $X(X\phi_1 + Y\phi_2) =0$ and $Y(X\phi_1 + Y\phi_2) = 0$, which happens if and only if
\[X\phi_1 + Y \phi_2 = C,\]
for some constant $C$.
\end{example}
In the special case of a Riemannian foliation, we can write the above as follows. Let $M$ be a simply connected Riemannian manifold with metric $\widetilde \tensorg$. Let $V = \spn Z$ be an integrable subbundle of $TM$ with corresponding foliation $\calF$ assumed to be Riemannian. By normalizing $Z$, we may suppose that $\|Z \|_{\widetilde \tensorg} = 1$. Let $N$ be the mean curvature vector field of $\calF$, defined as $N = \II(Z,Z)$ with $\II$ being the second fundamental form.
Let $\calR$ be the curvature of $D$ with respect to $V$, i.e., the vector-valued two-form defined by $\calR(X, Y) = \pr_V [\pr_D X, \pr_DY] .$ Assume that this never vanishes. For an arbitrary function $f$ define the gradient $\nabla f$ and horizontal gradient $\nabla^D$ by respectively
\[df(X) = \widetilde \tensorg(X, \nabla f) , \qquad \qquad df(\pr_D X) = \widetilde \tensorg(X, \nabla^D f).\]
Then we have the following identity.
\begin{corollary}
We can find a Riemannian metric $\tensorg$ such that $D$ and $V$ are orthogonal and $\calF$ is totally geodesic if and only if there exists $f \in C^\infty(M)$ such that
\[N = - \nabla^D \log \|\calR\|_{\widetilde \tensorg} + \nabla f.\]
\end{corollary}
\begin{proof}
Since $\calR$ is non-vanishing, $D$ is bracket-generating and equiregular step~$2$. Let $\flat\colon TM \to T^*M$ be the bijection $v \mapsto \widetilde \tensorg(v, \cdot)$. We write its inverse as $\sharp\colon T^*M \to TM$ and use the same symbol for the identifications of $\bigwedge^k TM$ and $\bigwedge^k T^*M$ through~$\widetilde \tensorg$. Define $\tau = \flat Z$. For every $X \in\Gamma(TM)$, one has
\[(\calL_Z \tau)(X) = \tensorg(Z, [\pr_D X, Z]) = - \frac{1}{2} (\calL_{\pr_D X} \widetilde \tensorg)(Z,Z) = \widetilde \tensorg(X, \II(Z,Z)) = \flat N(X).\]
Observe furthermore that, for any $X, Y \in \Gamma(D)$, we have
\[d\calL_Z \tau(X,Y) =  \calL_Z d\tau(X,Y) = (\calL_Z \pr_D^* d\tau)(X,Y),\]
since
\begin{linenomath}
\begin{align*} & \calL_Z (\id - \pr_D^*)d\tau(X,Y) = - d\tau(\pr_V [Z,X], Y) - d\tau(X, \pr_V[Z,Y]) \\
& = - d\tau(Z, Y) \tau([Z,X]) - d\tau(X, Z) \tau([Z,Y]) = 0.\end{align*}
\end{linenomath}

Define a selector $ \tau \otimes \chi$ of $D$ by
\[\chi = - \frac{1}{\| \pr^*_D d\tau\|^2} \sharp \pr_D^* d\tau.\]
Note that
\[d\calL_Z(\chi) = -\frac{1}{\| \pr_D^* d\tau\|^2} (\calL_Z \pr_D^* d\tau)(\sharp \pr_D^* \tau) =  -\frac{1}{\| \pr_D^* d\tau\|^2} \widetilde \tensorg(\calL_Z \pr_D^* d\tau, \pr_D^* \tau).\]
Furthermore, since $\calF$ is a Riemannian foliation, one gets
\[\frac{1}{\| \pr_D^* d\tau\|^2} \widetilde \tensorg(\calL_Z \pr_D^* d\tau, \pr_D^* \tau) =\frac{Z \| \pr_D^* d\tau\|^2}{2\| \pr_D^* d\tau\|^2}  = Z \log \|\pr_D^* d\tau\|.\]
Using that $\|\pr_D^* d\tau\| = \| \calR\|$, we apply Proposition~\ref{prop:Zfoliation} to deduce that there exists a metric making $D$ and $V$ orthogonal and $\calF$ totally geodesic if and only if
\begin{linenomath}
\begin{align*} 0 & = d(\flat N - \pr_V^* d\log \|\calR\|) = d(\flat N - \pr_V^* d\log \|\calR\| + d \log \|\calR\| )\\
& = d(\flat N + \pr_D^* d\log \| \calR\|).  \end{align*}
\end{linenomath}
Since $M$ is simply connected, there exists a smooth function $f$ on $M$ such that $\flat N + \pr_D^* d\log \|\calR\| = df$. Apply $\sharp$ to get the desired conclusion.
\end{proof}

\begin{remark}
We could have reached the conclusion of Proposition~\ref{prop:Zfoliation} without using horizontal holonomy as well. The argument goes as follows.
Given any Riemannian metric such that $D$ and $V$ are orthogonal, if $V$ is spanned by a unit-length vector field $\widetilde Z$ and $\widetilde \tau$ denotes the unique one-form verifying
$\widetilde \tau (D) = 0$ and $\widetilde \tau (\widetilde Z) = 1$, then $\tensorg(X, \II(\widetilde Z, \widetilde Z)) =( \calL_{\widetilde Z} \widetilde \tau)(X)$. Hence, finding such a basis vector field $\tilde Z$ such that $\calL_{\tilde Z} \tilde \tau = 0$ is equivalent to showing the existence of a metric making $\calF$ totally geodesic and orthogonal to $D$. We prove that the existence of $\tilde Z$ and \eqref{1dimdchi}.

Assume there exists such a vector field $\widetilde Z$. For every smooth function $f$, set  $Z = Z_f = e^f \widetilde Z$ and $\tau = e^{-f} \widetilde \tau$.  One deduces that
\[\calL_Z \tau = e^{-f} \iota_Z (- df \wedge \widetilde \tau + d \widetilde \tau) = \pr_D^* df.\]
Since this form coincides with $df$ on $D$, we have by Lemma~\ref{lemma:Selector}~(d) that
\[d_{\tau \otimes \chi} \pr_D^* df = d_{\tau \otimes \chi} df = 0.\]
Conversely, if $d_{\tau \otimes \chi} \calL_Z \tau = 0$ then again by Lemma~\ref{lemma:Selector}~(d), there exists a one-form $\beta$ such that $\beta|_D = \calL_Z \tau |_D$ and $d\beta = 0$. However, since $M$ is simply connected, then $\beta = df$ for some function $f$. As $\calL_Z \tau (Z) = 0$, it follows that $\calL_Z \tau = \pr_D^* df.$ The vector field $\widetilde Z = e^{-f} Z$ consequently has the desired properties.
\end{remark}

\subsection{Lie groups}

Let $G$ be a connected Lie group with a connected subgroup $K$. Let $\mathfrak{g}$ and $\mathfrak{k}$ be their respective Lie algebras. Let $\mathfrak{p}$ be any subspace such that $\mathfrak{g} = \mathfrak{p} \oplus \mathfrak{k}$ and define $D$ and $V$ by left translation of $\mathfrak{p}$ and $\mathfrak{k}$ respectively. The subbundle $V$ is then integrable with corresponding foliation $\calF = \{ a \cdot K \, : \, a \in F\}$. We again try to determine if there exists a Riemannian metric on $G$ such that $D$ and $V$ are orthogonal and $\calF$ is a totally geodesic foliation.

We use the same notation for an element in $\mathfrak{g}$ and its corresponding left invariant vector field. Consider the connection $\nabla$ on $V$ defined by
\[\nabla_A C := \pr_V [A,C] , \qquad \text{for any  $A \in \mathfrak{g}$, $C \in \mathfrak{k}$.}\]
If $A,B \in \mathfrak{g}$ and $C \in \mathfrak{k}$, the curvature of $\nabla$ is given by
\begin{linenomath}
\begin{align} \label{RLeftInv} R^\nabla(A,B)C & = \pr_V\left( [A, \pr_V [B,C]] - [B, \pr_V [A,C]] - [[A, B],C] \right) \\ \nonumber
& = \pr_V\left( [A, \pr_D [B,C]] - [B, \pr_D [A,C]]  \right).
\end{align}
\end{linenomath}
Introduce the connection $\omega$ on $\GL(\Sym^2 V^*)$ corresponding to $\nabla$.
We next provide a positive answer to the previous question in particular cases.
\begin{enumerate}[\rm (a)]
\item If $K$ is a normal subgroup, i.e., if $\mathfrak{k}$ is an ideal, then $R^{\nabla} = 0$, and therefore $\Hol^{\omega,D} \subseteq \Hol^{\omega}$ which reduces to the identity element. It follows that any inner product on $\mathfrak{k}$ can be extended to a Riemannian metric making $D$ orthogonal to $V$ and the foliation of $V$ totally geodesic. Since
\[V|_a = a \cdot \mathfrak{k} = \mathfrak{k} \cdot a, \qquad a \in G,\]
the desired Riemannian metric $\tensorg$ is on the form $\tensorg = \pr_D^* \tensorg^D + \pr_V^* \tensorg^V$, where $\tensorg^D$ is an arbitrary metric on $D$, while $\tensorg^V$ is the right translation of any inner product on $\mathfrak{k}$.
\item Assume that $[\mathfrak{k}, \mathfrak{p}] \subseteq \mathfrak{p}$. Then, one necessarily has that $\mathfrak{p} + [\mathfrak{p}, \mathfrak{p}] = \mathfrak{g}$, i.e., $D$ is equiregular of step~$2$. Furthermore, for any $A,B \in\mathfrak{g}$, $C \in \mathfrak{k}$, one has
\[ R^\nabla(A,B) C  = -[ \pr_{\mathfrak{k}} [\pr_{\mathfrak p} A, \pr_{\mathfrak p} B], C].\]
Let $\chi$ be any selector of $D$ and notice that $R^{\nabla}(\chi(A)) C = - [\pr_{\mathfrak k} A, C]$ for any $A \in \mathfrak{g}.$ Define $\eta^0 = R^{\nabla}$ and $\eta^1 = L^\nabla \iota_\chi \eta^0$. Using \eqref{EtaK}, we get $\eta^1 = - R^{\nabla}$, and, as a consequence, $R^\nabla_{\chi} = 0$. This reflects the fact that if we give $V$ a metric by left translation of any inner product on $\mathfrak{k}$ and extend this metric to $TM$ in an arbitrary way such that $D$ is orthogonal to $V$, then $\calF$ is a totally geodesic foliation.
\end{enumerate}
The general case is more complicated and is left for future research.

\subsection{An example with Carnot groups}
A Carnot group of step $r$ is a simply connected nilpotent Lie group with a Lie algebra $\mathfrak{g}$ with a given decomposition $\mathfrak{g} = \mathfrak{g}_1 \oplus \mathfrak{g}_2 \oplus \cdots \oplus \mathfrak{g}_r$ satisfying $[\mathfrak{g}_1, \mathfrak{g}_k] = \mathfrak{g}_{k+1}$ for $1 \leq k \leq r-1$ and $[\mathfrak{g}_1, \mathfrak{g}_r] = 0$.

For such a group we have the following result, where $\lfloor x \rfloor$ denotes the floor function, i.e.
\[\lfloor x \rfloor = \max \{ n \in \mathbb{Z} \, \colon \, x \geq n\}.\]
\begin{proposition} \label{prop:Carnot}
Define
\[\mathfrak{p}_1 = \bigoplus_{k \, \mathrm{odd}} \mathfrak{g}_k, \qquad \mathfrak{p}_2 =
\bigoplus_{\begin{subarray}{c} k \, \mathrm{even}\\  k \leq \lfloor r/2 \rfloor \end{subarray}} \mathfrak{g}_k, \qquad \mathfrak{k} =\bigoplus_{\begin{subarray}{c} k \, \mathrm{even}\\  k > \lfloor r/2 \rfloor \end{subarray}} \mathfrak{g}_k,\]
Let $D$ and $V$ be subbundles of $TG$ obtained by left translation of $\mathfrak{p}_1 \oplus \mathfrak{p}_2$ and $\mathfrak{k}$, respectively.
Then there exists a Riemannian metric on $M$ such that $D$ and $V$ are orthogonal and the foliation of $V$ is totally geodesic if and only if $[\mathfrak{p}_2, \mathfrak{k}] =0$. Furthermore, if $K$ is the connected subgroup of $G$ corresponding to~$\mathfrak{k}$, then the Ehresmann connection~$D$ on $\pi\colon G \to G/K$ can be made principal under a multiplication on the fibers of~$\pi$ if and only if $[\mathfrak{p}_2, \mathfrak{k}] = 0$.
\end{proposition}
\begin{proof}
Observe first that the following relations hold true
\begin{equation} \label{bracketsp1p2k} \begin{array}{lll}
{[\mathfrak{p}_1, \mathfrak{p}_1]} = \mathfrak{p}_2 \oplus \mathfrak{k}, \qquad  & [\mathfrak{p}_1, \mathfrak{p}_2] \subseteq \mathfrak{p}_1, \qquad &  [\mathfrak{p}_1, \mathfrak{k} ] \subseteq \mathfrak{p}_1,\\
{[\mathfrak{p}_2, \mathfrak{p}_2]} \subseteq  \mathfrak{p}_2 \oplus \mathfrak{k}, \qquad
& [\mathfrak{p}_2, \mathfrak{k}] \subseteq \mathfrak{k}, \qquad
& [\mathfrak{k}, \mathfrak{k}] = 0,\end{array}\end{equation}
It follows that $D$ is equiregular of step 2. Define a connection $\nabla$ on $V$ as follows. If $X \in\Gamma(TM)$ and $C \in \mathfrak{k}$, then set
\[\nabla_X C = \pr_V [\pr_D X, C].\]
From \eqref{bracketsp1p2k}, we obtain that for any $A \in \mathfrak{g}$, we have
\[\nabla_A C = [\pr_{\mathfrak{p}_2} A, C].\]
If follows that the curvature $R^{\nabla}$ is given for any $C \in \mathfrak{k}$ by
\[R^{\nabla}(A,B)C = \left\{ \begin{array}{ll} \nabla_{\pr_{\mathfrak{p}_2} [B,A]}C & \text{if } A, B \in \mathfrak{p}_1  \\
0 & \text{if } A \in \mathfrak{g},  B \in \mathfrak{p}_2 \oplus \mathfrak{k},
 \end{array} \right.\]
By using the definition of a Carnot group, we can define a selector $\chi$ of $D$ such that, if $C \in \mathfrak{g}_k \subseteq \mathfrak{k}$ for $k$ even and greater than $\lfloor r /2 \rfloor$, one has $\chi(C) = \sum_{j=1}^n A_i \wedge B_i$ with $A_i \in \mathfrak{g}_1 \subseteq \mathfrak{p}_1$ and $B_i \in \mathfrak{g}_{k-1} \subseteq \mathfrak{p}_1$. Since $\pr_{\mathfrak{p}_2} [A_i, B_i] =0$, we get $\iota_{\chi} R^{\nabla} = 0$, and hence $R^\nabla_\chi = R^{\nabla}$.

If $[\mathfrak{p}_2, \mathfrak{k}] = 0$, then $R^\nabla_\chi =0$. If $[\mathfrak{p}_2, \mathfrak{k}] \neq 0$, let $A \in \mathfrak{p}_2$ and $C \in \mathfrak{k}$ be any pair of elements such that $[A, C]$ is not zero. By replacing $C$ with $\ad(A)^k C$ for an appropriate value of $k \geq 0$, we may assume that $[A,[A,C]] = 0$. Write
\[A = \sum_{j=1}^k [B^1_j , B_j^2], \qquad B^1_j, B_j^2 \in \mathfrak{p}_1.\]
Then for any metric $\tensorg^V$ on $V$, we have
\[\sum_{j=1}^n R^\nabla_\chi(B_j^1, B_j^2) (\tensorg^V)(C, [A,C]) = - \tensorg^V([A,C],[A,C]) < 0.\]
Hence $R^{\nabla}_\chi(\cdot, \cdot)\tensorg^V$ is not equal to zero for any metric $\tensorg^V$ on $V$. One thus concludes by using Theorem~\ref{th:TGF}.

As for the second statement of principal bundle structure on $\pi$, $R^{\nabla}$ vanishes on~$V$ if and only if~$[\mathfrak{p}_2, \mathfrak{k}] = 0$. If the latter holds, then all left invariant vector fields~$C$ with values in $V$ satisfy $\nabla C = 0$ and since these are complete, the rest follows from Theorem~\ref{th:SUB}.
\end{proof}

\begin{example}
On $\mathbb{R}^2$ with coordinates $(x,y)$, define the vector fields
\[A = \frac{\partial}{\partial x}, \qquad B_k = \frac{1}{k!} x^k \frac{\partial}{\partial y}, \qquad 0 \leq k \leq n.\]
Note that $[A , B_{k+1}] = B_k$ for any $k \geq 0$. Define $\mathfrak{g} = \mathfrak{g}_1 \oplus \cdots \oplus \mathfrak{g}_{n+1}$ where
\[\mathfrak{g}_1 = \spn \{A,B_{n} \}, \quad \mathfrak{g}_k = \spn\{ B_{n+1-k},\}, \ 2\leq k\leq n+1 
.\]
Let $G$ be the corresponding simply connected Lie group of $\mathfrak{g}$. Since our Lie algebra was nilpotent, we will use (global) exponential coordinates, giving a point $a \in G$ coordinates $(r_0, r_1, \dots, r_n, s)$ if $a= \exp( \sum_{k=0}^n r_k B_k + s A )$.
\begin{enumerate}[\rm (a)]
\item Define $\mathfrak{k}$, $\mathfrak{p}_1$ and $\mathfrak{p_2}$ as in Propositions~\ref{prop:Carnot}. Then $[\mathfrak{p}_2, \mathfrak{k}] =0$, so if $D$ and $V$ are obtained by left translation of $\mathfrak{p}_1 \oplus \mathfrak{p}_2$ and $\mathfrak{k}$, respectively, $D$ is a principal connection on $G/K$.
\item Consider the Abelian subalgebra $\mathfrak{k} = \spn\{ B_1, \dots, B_{n-1}\}$ with complement $\mathfrak{p} = \spn\{ A, B_0, B_n\}.$ Let $G$ be a simply connected Lie group with Lie algebra~$\mathfrak{g}$ and let~$K$ be a subgroup with Lie algebra $\mathfrak{k}$. Let $D$ and $V$ be the subbundles of~$TG$ given by left translation of $\mathfrak{p}$ and $\mathfrak{k}$, respectively. Then $D$ is an Ehresmann connection on
\[\pi\colon G \to G/K,\]
but not principal, since $[\mathfrak{p},\mathfrak{k}] $ is not contained in $\mathfrak{p}$. We next determine a new multiplication on the fibers of $\pi$ so that $D$ becomes principal. For that purpose, we consider a connection $\nabla$ on $V$, such that if $A \in \mathfrak{g}$ and $C \in \mathfrak{k}$ are two left invariant vector fields, then
\[\nabla_A C =  \pr_V [\pr_D A, C].\]
This connection satisfies \eqref{verticalC} and it is simple to verify that $R^{\nabla} = 0$. Hence, the foliation of $V$ can be made totally geodesic with orthogonal complement $D$. Furthermore, $D$ can be made into a principal connection, as the vector fields $Z_1, \dots, Z_{n-1}$ defined by
$$Z_k = \sum_{j=0}^{k-1} \frac{(-1)^k s^k}{k!} B_{k-j} $$
are complete and satisfy $\nabla Z_k = 0$.
\end{enumerate}
\end{example}

\begin{example}
Let $F^{n,r}$ be the free nilpotent Lie group on $n$ generators of step $r$, i.e., the quotient of the free Lie group on $n$ generators by the subgroup corresponding to the ideal generated by the brackets of order $r$. Define $\mathfrak{k}$, $\mathfrak{p}_1$ and $\mathfrak{p}_2$ as in Proposition~\ref{prop:Carnot}. Then $[\mathfrak{p}_2, \mathfrak{k}] = 0$ if and only if $r < 8$. Hence, for $r \geq 8$, if $D$ and $V$ are defined by left translation of respectively $\mathfrak{p}_1 \oplus \mathfrak{p}_1$ and $\mathfrak{k}$, then there does not exist a Riemannian metric making $D$ and $V$ orthogonal and the foliation of $V$ totally geodesic.
\end{example}

\bibliographystyle{habbrv}
\bibliography{Bibliography}

\begin{thebibliography}{10}

\bibitem{AgSa04}
A.~A. Agrachev and Y.~L. Sachkov.
\newblock {\em Control theory from the geometric viewpoint}, volume~87 of {\em
  Encyclopaedia of Mathematical Sciences}.
\newblock Springer-Verlag, Berlin, 2004.
\newblock Control Theory and Optimization, II.

\bibitem{AmSi53}
W.~Ambrose and I.~M. Singer.
\newblock A theorem on holonomy.
\newblock {\em Trans. Amer. Math. Soc.}, 75:428--443, 1953.

\bibitem{BaBo15}
F.~Baudoin and M.~Bonnefont.
\newblock Curvature-dimension estimates for the {L}aplace--{B}eltrami operator
  of a totally geodesic foliation.
\newblock {\em Nonlinear Anal.}, 126:159--169, 2015.

\bibitem{BBG14}
F.~Baudoin, M.~Bonnefont, and N.~Garofalo.
\newblock A sub-{R}iemannian curvature-dimension inequality, volume doubling
  property and the {P}oincar\'e inequality.
\newblock {\em Math. Ann.}, 358(3-4):833--860, 2014.

\bibitem{BaGa11}
F.~{Baudoin} and N.~{Garofalo}.
\newblock {Curvature-dimension inequalities and Ricci lower bounds for
  sub-Riemannian manifolds with transverse symmetries}.
\newblock {\em To appear in: J. Eur. Math. Soc.}, Jan. 2011, Arxiv: 1101.3590.

\bibitem{BKW14}
F.~{Baudoin}, B.~{Kim}, and J.~{Wang}.
\newblock {Transverse Weitzenb\"ock formulas and curvature dimension
  inequalities on Riemannian foliations with totally geodesic leaves}.
\newblock {\em ArXiv e-prints}, Aug. 2014, 1408.0548.

\bibitem{bel96}
A.~Bella\"{\i}che.
\newblock The tangent space in sub-{R}iemannian geometry.
\newblock In A.~Bella\"{\i}che and J.-J. Risler, editors, {\em Sub-{R}iemannian
  Geometry}, Progress in Mathematics. Birkh\"{a}user, 1996.

\bibitem{BeBo82}
L.~B{\'e}rard-Bergery and J.-P. Bourguignon.
\newblock Laplacians and {R}iemannian submersions with totally geodesic fibres.
\newblock {\em Illinois J. Math.}, 26(2):181--200, 1982.

\bibitem{BlHe83}
R.~A. Blumenthal and J.~J. Hebda.
\newblock de {R}ham decomposition theorems for foliated manifolds.
\newblock {\em Ann. Inst. Fourier (Grenoble)}, 33(2):183--198, 1983.

\bibitem{BlHe84}
R.~A. Blumenthal and J.~J. Hebda.
\newblock Complementary distributions which preserve the leaf geometry and
  applications to totally geodesic foliations.
\newblock {\em Quart. J. Math. Oxford Ser. (2)}, 35(140):383--392, 1984.

\bibitem{Bri81}
F.~G.~B. Brito.
\newblock Une obstruction g\'eom\'etrique \`a l'existence de feuilletages de
  codimension {$1$} totalement g\'eod\'esiques.
\newblock {\em J. Differential Geom.}, 16(4):675--684, 1981.

\bibitem{Cai83}
G.~Cairns.
\newblock G\'eom\'etrie globale des feuilletages totalement g\'eod\'esiques.
\newblock {\em C. R. Acad. Sci. Paris S\'er. I Math.}, 297(9):525--527, 1983.

\bibitem{Cai86}
G.~Cairns.
\newblock A general description of totally geodesic foliations.
\newblock {\em Tohoku Math. J. (2)}, 38(1):37--55, 1986.

\bibitem{CGK15}
Y.~Chitour, M.~Godoy~Molina, and P.~Kokkonen.
\newblock Symmetries of the rolling model.
\newblock {\em Math. Z.}, 281(3-4):783--805, 2015.

\bibitem{CK-MSFM}
Y.~Chitour and P.~Kokkonen.
\newblock Rolling of manifolds and controllability in dimension three.
\newblock {\em To appear in Memoires de la SMF}, pages 1--118.

\bibitem{Cho39}
W.-L. Chow.
\newblock \"{U}ber {S}ysteme von linearen partiellen {D}ifferentialgleichungen
  erster {O}rdnung.
\newblock {\em Math. Ann.}, 117:98--105, 1939.

\bibitem{Elw14}
D.~Elworthy.
\newblock Decompositions of diffusion operators and related couplings.
\newblock In {\em Stochastic analysis and applications 2014}, volume 100 of
  {\em Springer Proc. Math. Stat.}, pages 283--306. Springer, Cham, 2014.

\bibitem{FGR97}
E.~Falbel, C.~Gorodski, and M.~Rumin.
\newblock Holonomy of sub-{R}iemannian manifolds.
\newblock {\em Internat. J. Math.}, 8(3):317--344, 1997.

\bibitem{Ge1993}
Z.~Ge.
\newblock Horizontal path spaces and carnot-carath\'{e}odory metrics.
\newblock {\em Pacific J. Math.}, 161(2):255--286, 1993.

\bibitem{Grasse1990}
K.~A. Grasse and H.~J. Sussmann.
\newblock Global controllability by nice controls.
\newblock In {\em Nonlinear controllability and optimal control}, volume 133 of
  {\em Monogr. Textbooks Pure Appl. Math.}, pages 33--79. Dekker, New York,
  1990.

\bibitem{GrTh15a}
E.~{Grong} and A.~{Thalmaier}.
\newblock {Curvature-dimension inequalities on sub-Riemannian manifolds
  obtained from Riemannian foliations: part I}.
\newblock {\em Math. Z.}, 2015, DOI:10.1007/s00209-015-1534-4, to appear in
  print.

\bibitem{GrTh15b}
E.~{Grong} and A.~{Thalmaier}.
\newblock {Curvature-dimension inequalities on sub-Riemannian manifolds
  obtained from Riemannian foliations: part II}.
\newblock {\em Math. Z.}, 2015, DOI:10.1007/s00209-015-1535-3, to appear in
  print.

\bibitem{HMCK15}
B.~{Hafassa}, A.~{Mortada}, Y.~{Chitour}, and P.~{Kokkonen}.
\newblock {Horizontal Holonomy for Affine Manifolds}.
\newblock {\em J. Dynam. Control Systems}, Nov. 2015,
  DOI:10.1007/s10883-015-9274-7, to appear in print.

\bibitem{Her60}
R.~Hermann.
\newblock A sufficient condition that a mapping of {R}iemannian manifolds be a
  fibre bundle.
\newblock {\em Proc. Amer. Math. Soc.}, 11:236--242, 1960.

\bibitem{Jea14}
F.~Jean.
\newblock {\em Control of Nonholonomic Systems: from Sub-Riemannian Geometry to
  Motion Planning}.
\newblock SpringerBriefs in Mathematics. Springer International Publishing,
  2014.

\bibitem{KPP97}
T.~H. Kang, H.~K. Pak, and J.~S. Pak.
\newblock Laplacian on a totally geodesic foliation.
\newblock {\em J. Geom.}, 60(1-2):74--79, 1997.

\bibitem{KoNo63}
S.~Kobayashi and K.~Nomizu.
\newblock {\em Foundations of differential geometry. {V}ol {I}}.
\newblock Interscience Publishers, a division of John Wiley \& Sons, New
  York-London, 1963.

\bibitem{KMS93}
I.~Kol{\'a}{\v{r}}, P.~W. Michor, and J.~Slov{\'a}k.
\newblock {\em Natural operations in differential geometry}.
\newblock Springer-Verlag, Berlin, 1993.

\bibitem{Mon93}
R.~Montgomery.
\newblock Generic distributions and {L}ie algebras of vector fields.
\newblock {\em J. Differential Equations}, 103(2):387--393, 1993.

\bibitem{Nag02}
P.-A. Nagy.
\newblock Nearly {K}\"ahler geometry and {R}iemannian foliations.
\newblock {\em Asian J. Math.}, 6(3):481--504, 2002.

\bibitem{Oze56}
H.~Ozeki.
\newblock Infinitesimal holonomy groups of bundle connections.
\newblock {\em Nagoya Math. J.}, 10:105--123, 1956.

\bibitem{Ras38}
P.~K. Rashevskii.
\newblock On joining any two points of a completely nonholonomic space by an
  admissible line.
\newblock {\em Math. Ann.}, 3:83--94, 1938.

\bibitem{Rei59}
B.~L. Reinhart.
\newblock Foliated manifolds with bundle-like metrics.
\newblock {\em Ann. of Math. (2)}, 69:119--132, 1959.

\bibitem{Rum94}
M.~Rumin.
\newblock Formes diff\'erentielles sur les vari\'et\'es de contact.
\newblock {\em J. Differential Geom.}, 39(2):281--330, 1994.

\bibitem{Sarychev1990}
A.~V. Sarychev.
\newblock Homotopy properties of the space of trajectories of a completely
  nonholonomic differential system.
\newblock {\em Dokl. Akad. Nauk SSSR}, 314(6):1336--1340, 1990.

\bibitem{Sul78}
D.~Sullivan.
\newblock A foliation of geodesics is characterized by having no ``tangent
  homologies''.
\newblock {\em J. Pure Appl. Algebra}, 13(1):101--104, 1978.

\end{thebibliography}

\end{document}